\documentclass[aap,preprint]{imsart}
\RequirePackage[OT1]{fontenc}
\RequirePackage{amsthm,amsmath}
\RequirePackage[numbers]{natbib}
\RequirePackage[colorlinks,citecolor=blue,urlcolor=blue]{hyperref}

\usepackage[utf8]{inputenc}
\usepackage{amsfonts, amsmath, amssymb, amsthm, bbm, color, enumerate, graphicx, mathtools, tikz, hyperref, relsize}
\usepackage[top=1in, bottom=1in, left=1in, right=1in]{geometry}

\newcommand{\one}{\mathbf{1}}

\newtheorem{theorem}{Theorem}

\newtheorem{lemma}[theorem]{Lemma}

\let\plainqed\qedsymbol
\newcommand{\claimqed}{$\lrcorner$}

\hypersetup{
 colorlinks=true,
 linkcolor=blue,          
 citecolor=red,       
 filecolor=red,   
 urlcolor=red, 
 pdftitle={},
 pdfauthor={},
 pdfsubject={},
 pdfkeywords={},
 linktocpage=true
}

\usepackage{fancyhdr}
 
\usepackage{natbib}

\usepackage{amssymb}
\usepackage{amsmath}
\usepackage{amsthm}

\newcommand{{\LPC}}{\textbf{LPC}}

\newcommand{\x}{\mathbf{x}}
\newcommand{\y}{\mathbf{y}}
\newcommand{\z}{\mathbf{z}}
\newcommand{\D}{\mathbf{D}}
\newcommand{\0}{\mathbf{0}}
\newcommand{\bb}{\mathbf{b}}
\newcommand{\RBM}{\operatorname{RBM}}

\newtheorem{thm}{Theorem}

\newtheorem{rem}[thm]{Remark}

\begin{document}

\begin{frontmatter}

\title{Parameter and dimension dependence of convergence rates to stationarity for reflecting Brownian motions}
\runtitle{Convergence rates to stationarity for RBM}
\author{\fnms{Sayan} \snm{Banerjee}\ead[label=e1]{sayan@email.unc.edu}}
\and
\author{\fnms{Amarjit} \snm{Budhiraja}\ead[label=e2]{amarjit@unc.edu}}
\affiliation{University of North Carolina, Chapel Hill}
\runauthor{Banerjee and Budhiraja}
\address{Department of Statistics \\and Operations Research\\
353 Hanes Hall CB \#3260\\
University of North Carolina\\
Chapel Hill, NC 27599\\
\printead{e1}\\
\printead{e2}}

\begin{abstract}
We obtain rates of convergence to stationarity in $L^1$-Wasserstein distance for a $d$-dimensional reflected Brownian motion (RBM) in the nonnegative orthant that are explicit in the  dimension and the system parameters. The results are then applied to a class of RBMs considered in \cite{BC2016} and to rank-based diffusions including the Atlas model. 
In both cases, we obtain explicit rates and bounds on relaxation times.
In the first case we improve the relaxation time estimates of $O(d^4(\log d)^2)$ obtained in \cite{BC2016} to $O((\log d)^2)$. In the latter case, we give the first results on explicit parameter and dimension dependent rates under the Wasserstein distance. The proofs do not require   an explicit form for the stationary measure or reversibility of the process with respect to this measure, and cover settings where these properties are not available. In the special case of the standard Atlas model \cite{Fern}, we obtain a bound on the relaxation time of $O(d^6(\log d)^2)$.
\end{abstract}

\begin{keyword}[class=MSC]
\kwd[Primary]{ 60J60}
\kwd{60H10}
\kwd[; secondary]{60J55}
\kwd{60K25}
\end{keyword}
\begin{keyword}
\kwd{Reflected Brownian motion}
\kwd{local time}
\kwd{coupling}
\kwd{Wasserstein distance}
\kwd{relaxation time}
\kwd{ heavy traffic}
\kwd{Atlas model}
\end{keyword}

\end{frontmatter}

\section{Introduction}
A $d$-dimensional obliquely reflected Brownian motion with drift in the nonnegative orthant plays a central role in Queuing Theory where it arises as a diffusion limit of scaled queue length processes  when the system is in the {\em heavy traffic regime} (namely the arrival rate and the service rate are approximately equal) \cite{reiman1977queueing, harrison1981reflected,halfin1981heavy,williams1998diffusion,bramson2001heavy}. Such a process is also used to describe the behavior of rank-based diffusions, namely a system of particles whose trajectories are given by Brownian motions with drift, where the drift and diffusion coefficients of a given particle at any given time depend on its relative rank in the system at that time. These models appear frequently in mathematical finance, eg. the Atlas model \cite{PP,IPS,BFK}. There has been extensive work in the study of stability of such reflected Brownian motions (RBM) that gives explicit sufficient conditions for positive recurrence for the RBM and the corresponding queuing systems \cite{harrison1987brownian, dupuis1994lyapunov,atar2001positive, dai1995positive,stolyar1995stability}. In this work, we obtain explicit exponential convergence rates (in Wasserstein distance) to equilibrium for multidimensional reflected Brownian motion (RBM) under a key stability condition identified in \cite{harrison1987brownian} (see Assumption (A2)). 
This assumption is known to be `almost necessary' for stability (see Remark \ref{rem:rem1} for a precise statement).  
The convergence rates  obtained in this work are {\em explicit (up to some universal constants) in the dimension and system parameters}. The system parameters are given by the drift vector, the covariance matrix of the Brownian motion, and the reflection matrix. Stationary distributions of RBM are rarely explicit and the convergence rates of the form obtained in this work provide important information for the construction of numerical schemes that sample from these stationary distributions.

There has been some prior work in this area. Exponential ergodicity was proved in \cite{BL07} for  semimartingale reflecting Brownian motions under the stability condition of \cite{dupuis1994lyapunov}.
This class includes  RBM of the form considered in this work. The paper \cite{BL07}  also established exponential ergodicity of certain reflected diffusions with 
state dependent drift and diffusion coefficients. 
The key ingredient in the proof was the construction of a suitable Lyapunov function along with establishing a minorization condition on a sufficiently large compact set (referred to as a `small set'). The Lyapunov function provides good control on the exponential moments of the return times to the small set while the minorization condition implies the existence of abstract couplings of two copies of the process (via construction of `pseudo-atoms' as described in Chapter 5 of \cite{meyn2012markov}) which have a positive chance of coalescing inside the small set. Together, they furnish exponential rates of convergence (in a weighted total variation distance). However, due to the somewhat implicit treatment of the process inside the small set, the rates obtained by this method  shed little light on how they qualitatively  depend on the system parameters or the state dimension. 
The paper \cite{IPS} obtained explicit convergence rates for a class of reversible rank-based diffusions with explicit stationary measures using Dirichlet form techniques (which crucially use reversibility). 
See also the discussion in Section \ref{sec:atlmod}. The  convergence considered in \cite{IPS} corresponds to that of time averages of bounded functionals of the state process to the corresponding stationary values in probability (see Theorem 1 of \cite{IPS}), which is considerably weaker than the $L^1$-Wasserstein distance considered in the current work.
The setting of one-dimensional RBM was considered in \cite{wang2014measuring} where (among other results) an estimate on the spectral gap  was provided as a function of the drift and the diffusion coefficient. In a recent work, \cite{BC2016} obtained dimension dependent bounds on rates of Wasserstein convergence for a class of RBM. Under conditions on the drift vector, the covariance matrix of the Brownian motion, and the  reflection matrix (see Conditions (BC1)-(BC3) in Section \ref{eg}), \cite{BC2016} analyzed the behavior of the RBM inside the small set explicitly by considering \emph{synchronous couplings} (namely, couplings where the RBM starting from different points are driven by the same Brownian motion). Using explicit couplings to obtain better convergence rate estimates is a relatively recent but developing area. See \cite{bolley2010trend,Eberle2015,Eberle2017,eberle2016quantitative} for such results for other classes of diffusions. In this work, we revisit the idea of constructing synchronous couplings for RBM. Under quite general conditions (specifically, the ones introduced in \cite{harrison1987brownian} that guarantee  the existence of strong solutions and positive recurrence), we construct a suitable Lyapunov function and identify (an appropriate analogue of) a small set that both depend crucially on the process parameters and the state dimension. This, along with a careful treatment of excursions from the small set, enables us to quantify contraction rates in $L^1$-distance for synchronous couplings starting from distinct points and thereby obtain  rates of Wasserstein convergence that are given explicitly in terms  of the system parameters, the state dimension, and some  constants (that do not depend on dimension or model parameters). These convergence rates, together with bounds on relaxation times of the RBM that follow from it, are the main results  of this work and are given in Theorem \ref{wasthm}. In Section \ref{eg} we apply these results to the class of RBMs considered in \cite{BC2016} and rank-based diffusions considered in \cite{IPS}. In the former case, we substantially improve the relaxation time estimates from $O(d^4(\log d)^2)$ obtained in \cite{BC2016} to $O((\log d)^2)$. In the latter case, we give the first results on explicit parameter and dimension dependent rates under the Wasserstein distance. The proofs do not require   an explicit form for the stationary measure or reversibility of the process with respect to this measure, and cover settings where these properties are not available. In the special case of the standard Atlas model \cite{Fern}, we give a bound on the relaxation time of $O(d^6(\log d)^2)$ (see Remark \ref{sam}).

\section{Model, notation and assumptions}
Let $B$ be a $d$-dimensional standard Brownian motion and let  $\mu \in \mathbb{R}^d$ and $\D, R \in \mathbb{R}^{d \times d}$. Consider for $\x \in \mathbb{R}^d_+ := [0, \infty)^d$ the $\mathbb{R}^d_+$-valued continuous stochastic process given by the equation
\begin{equation}\label{RBMdef}
X(t;\x) = \x + \D B(t) + \mu t + RL(t),
\end{equation}
 where $L$, referred to as  the local time process, is a non-decreasing continuous process satisfying 
 \begin{equation}\label{loctim}
 L(0)=0, \;\; \int_0^t X_i(s;\x)dL_i(s) = 0 \mbox{ for all } t>0 \mbox{ and } 1 \le i \le d.
 \end{equation}
 We will make the following basic assumptions.\\\\\\

 \textbf{Assumptions: }
 \begin{itemize}
 \item[(A1)] The matrix $P := I - R^T$ is substochastic (non-negative entries and row sums bounded above by 1) and transient ($P^n \rightarrow 0$ as $n \rightarrow \infty$).  
 \item[(A2)] $\bb := -R^{-1}\mu > 0$.
 \item[(A3)] The matrix $\Sigma = \D\D^T$ is positive definite.
 \end{itemize}
 The paper \cite{harrison1981reflected} shows that under (A1) there is a unique strong solution to \eqref{RBMdef} -
 \eqref{loctim}, namely for each $\x \in \mathbb{R}^d_+$ there is a unique pair of continuous stochastic processes $(X,L)$ satisfying the above equations. 
 This assumption is satisfied by the routing matrix  of any single-class open queueing network  \cite{harrison1981reflected} and consequently diffusion limits of such networks can be characterized by \eqref{RBMdef} -
 \eqref{loctim}.
 The collection $\{X(\cdot; \x)\}_{\x \in \mathbb{R}^d_+}$ defines a strong Markov process (see \cite{harrison1987brownian}) which we denote as $\RBM(\mu, \Sigma, R)$ and   refer to simply as the reflected Brownian motion (RBM). 
 The matrix $R$ describes the reflection mechanism, specifically,
  the $i$-th column of $R$ gives the direction of reflection on the $i$-th face of the orthant. The conditions on $P$ in particular say that its spectral radius is strictly less than $1$.  The matrix $\Sigma = \D\D^T$ gives the covariance matrix associated with the diffusion term of \eqref{RBMdef}. \\

\textbf{Notation: } Although $\mu, \Sigma$ and $R$ depend on the dimension $d$, this dependence is suppressed to avoid cumbersome notation. We will write $\bb =- R^{-1}\mu$. The entries of $\bb$ will be denoted by $b_i$, and the diagonal entries of $\Sigma$ will be denoted by $\sigma_i^{2}$, where $1 \le i \le d$. {\em All constants appearing in the statements of lemmas and theorems will be universal in that they do not depend on model parameters or the dimension $d$, unless noted otherwise.}\\

\begin{rem}
	\label{rem:rem1}
	{\em Unique strong solutions of the RBM that follow from (A1) imply that
	  any coupling of the driving Brownian motions translate into a coupling of the processes themselves.
	 Throughout this work we will take the family} $\{X(\cdot; \x)\}_{\x \in \mathbb{R}^d_+}$ {\em to be } driven by the same Brownian motion, {\em namely we will consider a} synchronous coupling {\em of the processes starting from different initial conditions. 
	 Assumption (A2) is the well known `stability condition' which is sufficient for the existence of a stationary measure \cite{harrison1987brownian}. The condition is almost necessary for stability in that if} $\bb_i<0$ {\em for some} $i$ {\em then the RBM is transient \cite{buddupsimp}. Assuming (A3) in addition to (A1)-(A2) gives that the strong Markov process} $\RBM(\mu, \Sigma, R)$ {\em has a unique stationary probability distribution \cite{harrison1987brownian}.}
\end{rem}

\section{Main Result}\label{con}
Following \cite{BC2016}, define the following stopping times: $\eta^0(\x) = 0$ and
\begin{equation*}
\eta_i^k(\x) = \inf\{t \ge \eta^{k-1}(\x) + 1: X_i(t,\x) = 0\}, \ \ \ \eta^k(\x) = \sup\{\eta_i^k(\x): 1 \le i \le d\}.
\end{equation*}
Define
$$
\mathcal{N}(t;\x) = \sup\{ k \ge 0 : \eta^k(\x) \le t\}.
$$
Also define the {\em contraction coefficient}
\begin{equation}\label{nd}
n(R) := \inf\{ n \ge 1: \|P^n \mathbf{1}\|_{\infty} \le 1/2\},
\end{equation}
where $\mathbf{1}$ is a $d$-dimensional vector of ones and for $u \in \mathbb{R}^d$, $\|u\|_{\infty} := \sup_{1\le i \le d} |u_i|$.
By Assumption (A1), $n(R) < \infty$. This quantity plays a key role in quantifying the convergence rate to equilibrium.

We now present the main result of this work. Given probability measures $\mu$ and $\nu$ on $\mathbb{R}^d_+$, a probability measure $\gamma$ on $\mathbb{R}^d_+\times \mathbb{R}^d_+$
is said to be a coupling of $\mu$ and $\nu$ if $\gamma(\cdot \times \mathbb{R}^d_+) = \mu(\cdot)$ and $\gamma(\mathbb{R}^d_+ \times \cdot) = \nu(\cdot)$.
 The $L^1$-Wasserstein distance between two probability measures $\mu$ and $\nu$ on $\mathbb{R}^d_+$ is given by
$$
W_1(\mu, \nu) = \inf\left\lbrace\int_{\mathbb{R}^d_+ \times \mathbb{R}^d_+}\|\x-\y\|_1 \gamma(\operatorname{d}\x, \operatorname{d}\y) \ : \gamma \text{ is a coupling of } \mu \text{ and } \nu\right\rbrace,
$$
where for a vector $z \in \mathbb{R}^d$, $\|z\|_1 = \sum_{i=1}^d |z_i|$.
We will denote the law of a random variable $X$ by $\mathcal{L}(X)$. 
Recall that from \cite{harrison1987brownian}, under Assumptions (A1)-(A3), there is a unique stationary distribution
of the RBM.
Denote by $\mathbf{X}(\infty)$ a random vector sampled from this stationary distribution.
Define the \emph{relaxation time}, $t_{rel}(\x)$  for the RBM starting from $\x \in \mathbb{R}^d_+$ as 
$$t_{rel}(\x) := \inf\{t\ge 0: W_1\left(\mathcal{L}(X(t; \x)), \mathcal{L}(\mathbf{X}(\infty))\right) \le 1/2\}.$$
We will abbreviate the parameters of the RBM as $\Theta :=(\mu, \Sigma, R)$. Recall that these parameters are required to satisfy (A1)-(A3). We will quantify rate of convergence to equilibrium in terms of the following functions of $\Theta,d$. Fix $\kappa \in (0,\infty)$. Let
\begin{align*}
a(\Theta) &:= \sup_{1 \le i \le d}\left[\frac{\sum_{j=1}^d(R^{-1})_{ij}\sigma_j}{b_i}\right],\ \ \ \ \ \ \ 
b(\Theta) := \sup_{1\le i \le d}\left[\frac{\sum_{j=1}^d(R^{-1})_{ij}\sigma_j}{\sigma_i}\right],\\
R_1(\Theta,d) &:= n(R)(1+a(\Theta)^2\log(2d)),\ \ \ \ \ \ \ 
R_2(\Theta) := a(\Theta)^2b(\Theta),\\
C_1(\x,\Theta)&:= 2\|\x\|_1 + a(\Theta)\sum_{i,j}(R^{-1})_{ij}\sigma_j,\\
C_2(\x,\Theta, \kappa) &:= 2\|\x\|_1e^{3(\kappa a(\Theta)b(\Theta))^{-1}\|\x\|_{\infty}^*} + a(\Theta)\left[2d(1+d)\left(\sum_{i,j}(R^{-1})^2_{ij}\right)\left(\sum_{j=1}^d\sigma_j^2\right)\right]^{1/2}.
\end{align*}
 For any $\x \in \mathbb{R}^d_+$, define $\|\x\|_{\infty}^* := \sup_{1\le i \le d}\sigma_i^{-1}\x_i$.
\begin{theorem}\label{wasthm}
There exist a  $t_0 \in (0,\infty)$ and $D_1, D_2 \in (0,\infty)$ such that for every $d \in \mathbb{N}$, 
$\x \in \mathbb{R}^d_+$, every parameter choice $\Theta$, and $ t \ge t_0 \left(1 + (a(\Theta))^2\log(2d)\right)$,
\begin{align*}
	W_1\left(\mathcal{L}(X(t; \x), \mathcal{L}(\mathbf{X}(\infty))\right) &\le
\mathbb{E}(\|X(t;\x) - X(t; \mathbf{X}(\infty))\|_1)\\
& \le C_1(\x,\Theta)\left(2e^{-\frac{D_1t}{R_1(\Theta,d)}} + e^{-\frac{t}{16D_2R_2(\Theta)}}\right)
 + C_2(\x,\Theta, D_2)e^{-\frac{t}{8D_2R_2(\Theta)}}.
\end{align*}
 In particular, the relaxation time satisfies
\begin{align*}
t_{rel}(\x) \le \max\{D_1^{-1}R_1(\Theta,d) \log(8C_1(\x,\Theta)) + 16D_2R_2(\Theta)\log[4(C_1(\x,\Theta) + C_2(\x,\Theta, D_2))],\\
\qquad \qquad \qquad t_0 \left(1 + (a(\Theta))^2\log(2d)\right)\}.
\end{align*}
\end{theorem}

\begin{rem}
	The universal constants $t_0, D_1$ and $D_2$ will be identified in Sections \ref{sm} and \ref{excth}. Specifically,
$t_0$ and $D_1$ are introduced in Lemma \ref{wassest} (see \eqref{eq:rev1})	and 
 $D_2 := \max\{A_0, 9\}$,	where   $A_0$ is introduced in Lemma \ref{lyaphit}.
\end{rem}
\begin{rem}\label{betterconst}
The proof of Theorem \ref{wasthm} (see Remark \ref{bc}) will show that one can, in fact, obtain a better bound of the form
\begin{multline*}
	W_1\left(\mathcal{L}(X(t; \x), \mathcal{L}(\mathbf{X}(\infty))\right) \le
\mathbb{E}(\|X(t;\x) - X(t; \mathbf{X}(\infty))\|_1)\\
 \le C_1(\x,\Theta)\left(e^{-\frac{D_1n(R)t}{R_1(\Theta,d)}} + e^{-\frac{32(\log 2)D_1t}{R_1(\Theta,d)}} + e^{-\frac{t}{16D_2R_2(\Theta)}}\right)
 + C_2(\x,\Theta, D_2)e^{-\frac{t}{8D_2R_2(\Theta)}}.
\end{multline*}
This bound leads to a better choice of the universal constants appearing in the exponents of the bound when $n(R)$ is large. As a consequence the  bounds on relaxation times and the bounds given in the examples of Section \ref{eg} can be slightly improved using the above estimate. However, this improved bound leads to cumbersome expressions in the bounds and relaxation time estimates in Section \ref{eg}. Moreover, our main goal is to highlight the dependence of the convergence rates on system parameters which is completely captured by Theorem \ref{wasthm}. Hence, we do not give details on how the improved bound can be obtained, however see Remark \ref{bc} for some additional comments.
\end{rem}
An important ingredient in the proof is the following analogue of Lemma 3 from \cite{BC2016}
 which shows that the synchronous coupling gives an a.s. contraction of the $L^1$-distance $\|X(t;\x) - X(t;\0)\|_1$ which can be quantified as follows. The proof is similar to that in \cite{BC2016} and so only a sketch is provided.
\begin{lemma}[ see \cite{BC2016}]\label{contract}
For $\x \in \mathbb{R}^d_+$ and $t \ge0$,
$$
\|X(t;\x) - X(t;\0)\|_1 \le 2\|\x\|_1 2^{-\mathcal{N}(t;\x)/n(R)}.
$$
\end{lemma}
\begin{proof}
The main idea is to associate the substochastic matrix $P$ with a Markov chain on states $\{0,1,\dots,d\}$ absorbed at $0$ and show that $\|\x\|_1^{-1}\|X(t;\x) - X(t;\0)\|_1$ (assuming $\|\x\|_1 \neq 0$) is bounded above by the maximum over the initial state $i$ of the probability that, starting from $i$, the Markov chain is not absorbed by time $\mathcal{N}(t;\x)$. Using this idea, Lemma 2 in \cite{BC2016} and the proof of Lemma 3 in \cite{BC2016} establish
\begin{equation}\label{eq:blanchenest}
\|X(t;\x) - X(t;\0)\|_1 \le \|P^{\mathcal{N}(t;\x)} \mathbf{1}\|_{\infty} \|\x\|_1.
\end{equation}
The lemma now follows from the definition of $n(R)$ given in \eqref{nd} above.
\end{proof}
\begin{rem}\label{betterer}
The quantity $n(R)^{-1}$ defined in \eqref{nd} gives an explicit bound on the exponential decay rate of $\|P^{n} \mathbf{1}\|_{\infty}$ with $n$. Note that $n(R)$ possibly depends on the dimension $d$, but the dependence is solely through $R$. Sometimes (as we will see in the first example of Section \ref{eg}) it is possible to get a better bound in the sense that we can obtain positive constants $C(R,d)$ and $n'(R)< n(R)$ such that
$$
\|P^{n}\mathbf{1}\|_{\infty} \le C(R,d) 2^{-n/n'(R)},  \ \ n \ge 0.
$$
In this case, we can replace the bound in Lemma \ref{contract} by
\begin{equation}\label{eq:bettbd}
\|X(t;\x) - X(t;0)\|_1 \le C(R,d)\|\x\|_1 2^{-\mathcal{N}(t;\x)/n'(R)}.
\end{equation}
\end{rem}

\begin{rem}\label{wastwo}
In cases where we can obtain the better bound \eqref{eq:bettbd}, the constants
$R_1(\Theta,d)$ and $C_1(\x, \Theta)$ appearing in the bounds on Wasserstein distance and relaxation time  in Theorem \ref{wasthm} can be replaced by $R_1'(\Theta,d)$ and $C_1'(\x, \Theta,d)$ respectively, where
\begin{align*}
R_1'(\Theta,d) &:= n'(R)(1+a(\Theta)^2\log(2d)),\\
C_1'(\x,\Theta, d)&:= 2\|\x\|_1 + \frac{a(\Theta)C(R, d)}{2}\sum_{i,j}(R^{-1})_{ij}\sigma_j.
\end{align*}
\end{rem}

\section{Outline of Approach}
We now give an outline of our approach.
\begin{itemize}
\item[(i)] We use a key idea from \cite{BC2016} which shows that, under the synchronous coupling, the $L^1$-distance between the two processes $X(\cdot;\0)$ and $X(\cdot;\x)$ decreases with time. Using this idea, we provide an estimate on the rate of decay of this $L^1$-distance in terms of a `contraction coefficient' which quantifies the decay rate of $\|P^n\one\|_{\infty}$ with $n$. The precise statement was formulated as Lemma \ref{contract} in Section \ref{con}.
\item[(ii)] We use the fact that for any $v>0$ in $\mathbb{R}^d$ satisfying $R^{-1}v \le \bb$, one can dominate the process $X(\cdot;\x)$ in an appropriate manner by a normally reflected Brownian motion with drift $-v$ in $\mathbb{R}^d_+$. This process, written as $X^+_v(\cdot;\x)$, is technically simpler to analyze. The idea of dominating an $\RBM(\mu, \Sigma, R)$ by a normally reflected RBM is due to \cite{harrison1987brownian}. Next, we choose an appropriate compact set (which plays a role similar to the `small set' in the terminology of \cite{meyn2012markov}) such that one can obtain a tight control over return times to this set (this is done via Lyapunov function techniques in Lemma \ref{lyaphit}) and, loosely speaking, is such that the $L^1$-distance between the synchronously coupled processes $X(\cdot;\0)$ and $X(\cdot;\x)$ decreases by a constant factor each time the process $X^+_v(\cdot;\x)$ visits this set (this result is  formulated in Lemma \ref{zerohit}). A crucial ingredient here is the introduction of a suitable weighted norm (see \eqref{wnorm}) whose sub-level sets are the appropriate `small sets' with the desired contraction property. The
definition of this norm is guided by an analysis of how the maximum process for each coordinate scales with the system parameters. This weighted norm is  used to construct the small set and also an appropriate Lyapunov function. These constructions and their properties are studied in Section \ref{sm}.
\item[(iii)] In Section \ref{excth}, we obtain the rate of decay of $\|X(t;\x) - X(t;\0)\|_1$ with time $t$, in terms of the parameter $v$ of the dominating normally reflected RBM,
by decomposing the path of $X^+_v(\cdot;\x)$ into excursions from the small set obtained in (ii) and using the estimates from Section \ref{sm} for probabilities of certain events associated with these excursions.
\item[(iv)] Finally in Section \ref{optimize} we prove our main result, namely Theorem \ref{wasthm}, where we obtain explicit parameter and dimension dependent rates of decay in $L^1$-Wasserstein distance between the processes $X(\cdot;\0)$ and $X(\cdot;\x)$ with time $t$ by optimizing the rates derived in (iii) over  the parameter $v>0$ of the dominating RBM.
\end{itemize}

Before proceeding to the proof we apply Theorem \ref{wasthm} in two settings, the first is that of RBM satisfying the assumptions of \cite{BC2016} and the second corresponds to that of rank-based diffusions such as  the Atlas model.

\section{Examples}\label{eg}
We will use Theorem \ref{wasthm} (and Remark \ref{wastwo}) to obtain bounds on the rate of convergence to equilibrium in two examples that are discussed in Sections \ref{sec:bcrbm} and \ref{sec:atlmod} below.

\subsection{Blanchet-Chen RBM}
\label{sec:bcrbm}
This refers to the class of RBM under the set of assumptions in \cite{BC2016}, namely:
\begin{itemize}
\item[(BC1)] The matrix $P$ is substochastic and there exist $\kappa>0$ and $\beta \in (0,1)$ not depending on the dimension $d$ such that $\|\one^TP^n\|_{\infty} \le \kappa(1 - \beta)^n$ for all $n \ge 0$. 
\item[(BC2)] There exists $\delta>0$ independent of $d$ such that $R^{-1}\mu < -\delta \one$.
\item[(BC3)] There exists $\sigma>0$ independent of $d$ such that $ \sigma_i :=\sqrt{\Sigma_{ii}}$ satisfies $\sigma^{-1} \le \sigma_i \le \sigma$ for every $1 \le i \le d$.
\end{itemize}
Under the above conditions \cite{BC2016} give a polynomial bound  of $O(d^4(\log d)^2)$  on the relaxation time of the RBM.
As shown in the following theorem, Theorem \ref{wasthm} gives a  substantial improvement by establishing a polylogarithmic relaxation time of $O((\log d)^2)$. 
\begin{theorem}\label{BCbound}
Under Assumptions (BC1), (BC2) and (BC3), there exist positive constants $E_1, E_2, E_3, E_4,t_1$ such that for any $\x \in \mathbb{R}^d_+, t \ge t_1 \max\{\|\x\|_{\infty}, \log(2d)\}$,
\begin{equation*}
\mathbb{E}(\|X(t;\x) - X(t; \mathbf{X}(\infty))\|_1) \le  2\left(2\|\x\|_1 + E_1d^2\right)e^{-E_2 t/\log(2d)} + \left(4\|\x\|_1 + E_1d^2\right)e^{-E_4t/2} + E_3 d^{2}e^{-E_4t}.
\end{equation*}
In particular, the relaxation time satisfies
\begin{align*}
t_{rel}(\x) \le \max \left\lbrace E_2^{-1}\log\left[8\left(2\|\x\|_1 + E_1d^2\right)\right]\log (2d) + E_4^{-1}\left[2\log\left[8\left(4\|\x\|_1 + E_1d^2\right)\right] + \log (8E_3d^{2})\right],\right.\\
\left.\qquad \qquad \qquad  t_1 \max\{\|\x\|_{\infty}, \log(2d)\}\right\rbrace.
\end{align*}
\end{theorem}
\begin{proof}
Observe that
$$
\|P^n\one\|_{\infty} \le \one^TP^n\one \le d\|\one^TP^n\|_{\infty} \le d\kappa(1-\beta)^n.
$$
Thus, the hypothesis of Remark \ref{betterer} is satisfied with 
$$
C(R,d)=\kappa d,\ \ n'(R) \equiv n' :=  \frac{\log(2)}{\log (1-\beta)^{-1}} + 1.
$$
Now we will use Theorem \ref{wasthm} in conjunction with Remark \ref{wastwo}.

Under Assumptions (BC1), (BC2) and (BC3), we have the following bounds:
\begin{align*}
a(\Theta) &= \sup_{1 \le i \le d}\left[\frac{\sum_{j=1}^d(R^{-1})_{ij}\sigma_j}{b_i}\right] \le \frac{\|R^{-1}\one\|_{\infty}\sigma}{\delta} \le \frac{\kappa \sigma}{\beta \delta},\\
b(\Theta) &= \sup_{1\le i \le d}\left[\frac{\sum_{j=1}^d(R^{-1})_{ij}\sigma_j}{\sigma_i}\right] \le \frac{\|R^{-1}\one\|_{\infty}\sigma}{\sigma^{-1}} \le \frac{\kappa \sigma^2}{\beta},\\
R_1'(\Theta,d) &= n'(R)(1+a(\Theta)^2\log(2d)) \le n'\left(1 + \frac{\kappa^2\sigma^2}{\beta^2\delta^2}\log(2d)\right),\\
R_2(\Theta) &= a(\Theta)^2b(\Theta) \le \frac{\kappa^3\sigma^4}{\beta^3\delta^2},\\
C_1'(\x,\Theta, d)&= 2\|\x\|_1 + \frac{a(\Theta)C(R,d)}{2}\sum_{i,j}(R^{-1})_{ij}\sigma_j \le 2\|\x\|_1 + \frac{\kappa^2 \sigma d^2}{2\beta \delta}\|R^{-1}\one\|_{\infty}\sigma \le 2\|\x\|_1 + \frac{\kappa^3 \sigma^2d^2}{2\beta^2 \delta},
\end{align*}
where we have used the observation that under Assumption (BC1), one has $\|R^{-1}\one\|_{\infty} \le \kappa/\beta$.
Next, observe that $\|\x\|_{\infty}^* \le \sigma\|\x\|_{\infty}$. This, along with the bound on $a(\Theta)$ obtained above, implies that for $t \ge \frac{48\kappa \sigma^2}{\beta \delta}\|\x\|_{\infty}$, $3(D_2a(\Theta)b(\Theta))^{-1}\|\x\|_{\infty}^* \le t/(16D_2R_2(\Theta))$. Hence, for such $t$,
\begin{align*}
 C_2(\x,\Theta, D_2)e^{-\frac{t}{8D_2R_2(\Theta)}} &= 2\|\x\|_1e^{3(D_2a(\Theta)b(\Theta))^{-1}\|\x\|_{\infty}^*}e^{-\frac{t}{8D_2R_2(\Theta)}}\\
 &\quad + a(\Theta)\left[2d(1+d)\left(\sum_{i,j}(R^{-1})^2_{ij}\right)\left(\sum_{j=1}^d\sigma_j^2\right)\right]^{1/2}e^{-\frac{t}{8D_2R_2(\Theta)}}\\
 &\le 2\|\x\|_1e^{-\frac{t}{16D_2R_2(\Theta)}}  + \frac{\kappa \sigma}{\beta \delta}\left[2d(1+d)d\left(\|R^{-1}\one\|_{\infty}\right)^2\left(d\sigma^2\right)\right]^{1/2}e^{-\frac{t}{8D_2R_2(\Theta)}}\\
 &\le 2\|\x\|_1e^{-\frac{t}{16D_2R_2(\Theta)}}  + \frac{2\kappa^2 \sigma^2}{\beta^2 \delta}d^{2}e^{-\frac{t}{8D_2R_2(\Theta)}}\\
 & \le 2\|\x\|_1e^{-\frac{\beta^3\delta^2t}{16D_2\kappa^3\sigma^4}}  + \frac{2\kappa^2 \sigma^2}{\beta^2 \delta}d^{2}e^{-\frac{\beta^3\delta^2t}{8D_2\kappa^3\sigma^4}}.
\end{align*}
Take $E_1 =\frac{\kappa^3\sigma^2}{2\beta^2 \delta}, E_2 = D_1\left[n'\left(2 + \frac{\kappa^2\sigma^2}{\beta^2\delta^2}\right)\right]^{-1}, E_3 =  \frac{2\kappa^2 \sigma^2}{\beta^2 \delta}, E_4 = \frac{\beta^3\delta^2}{8D_2\kappa^3\sigma^4}$. Using the above bounds in Theorem \ref{wasthm} (modified as in Remark \ref{wastwo}), for any $\x \in \mathbb{R}^d_+$, $t \ge \max\left\lbrace t_0 \left(1 + \frac{\kappa^2\sigma^2}{\beta^2\delta^2}\log(2d)\right), \frac{48\kappa \sigma^2}{\beta \delta}\|\x\|_{\infty}\right\rbrace$,
\begin{multline*}
\mathbb{E}(\|X(t;\x) - X(t; \mathbf{X}(\infty))\|_1) \le C_1'(\x,\Theta,d)\left(2e^{-\frac{D_1t}{R_1'(\Theta,d)}} + e^{-\frac{t}{16D_2R_2(\Theta)}}\right)
 + C_2(\x,\Theta, D_2)e^{-\frac{t}{8D_2R_2(\Theta)}}\\
 \le \left(2\|\x\|_1 + E_1d^2\right)\left(2e^{-E_2 t/\log(2d)} + e^{-E_4t/2}\right) + 2\|\x\|_1e^{-E_4t/2} + E_3 d^{2}e^{-E_4t}.
\end{multline*}
This proves the first part of the theorem upon taking $t_1 = \max\left\lbrace t_0 \left(1 + \frac{\kappa^2\sigma^2}{\beta^2\delta^2}\right), \frac{48\kappa \sigma^2}{\beta \delta}\right\rbrace$. The bound on the relaxation time follows immediately from the first part.
\end{proof}

\subsection{Gap process of rank-based diffusions} 
\label{sec:atlmod}
Rank based diffusions are interacting particle systems where the drift and diffusion coefficient of each particle depends on its rank. Mathematically, they are represented by the SDE:
\begin{equation}\label{rbdef}
dX_i(t) = \left(\sum_{j=1}^{d+1} \delta_j \one_{[X_i(t) = X_{(j)}(t)]}\right)dt + \left(\sum_{j=1}^{d+1} \sigma_j \one_{[X_i(t) = X_{(j)}(t)]}\right)dW_i(t)
\end{equation}
for $1\le i \le d+1$, where $\{X_{(j)}(t) : t \ge 0\}$ denotes the trajectory of the rank $j$ particle as a function of time $t$ ($X_{(1)}(t) \le \dots\le X_{(d+1)}(t) \text{ for all } t \ge 0$), $\delta_j, \sigma_j$ denote the drift and diffusion coefficients of the rank $j$ particle, and
$W_i$, $1\le i \le d$, are mutually independent standard one dimensional Brownian motions.
We will assume throughout that $\sigma_i>0$ for all $1 \le i \le d+1$. Rank-based diffusions have been proposed and extensively studied as models for problems in finance and economics. A special case is the  Atlas model \cite{Fern}  where the minimum particle (i.e. the particle with rank $1$) is a Brownian motion with positive drift and the remaining particles are Brownian motions without drift (i.e. $\delta_i=0$ for all $i>1$). The general setting considered in 
\eqref{rbdef} was introduced in 
\cite{BFK}. 
In order to study the long time behavior, it is convenient to consider the gap process $Y = (Y_1,\dots,Y_{d})$, given by $Y_i = X_{(i+1)} - X_{(i)}$ for $1 \le i \le d$. 
The process $Y \equiv Y(t;\y)$ is a RBM in $\mathbb{R}_+^{d}$ given as
\begin{equation*}
Y(t;\y) = \y + \D B(t) + \mu t +  RL(t)
\end{equation*}
where $\y$ is the initial gap sequence, $B$ is a standard $d$-dimensional Brownian motion, $\mu_i = \delta_{i+1} - \delta_i$ for $1 \le i \le d$, $\D \in \mathbb{R}^{d \times d}$, $L$ is the local time process 
associated with $Y$ and $R$ satisfies Assumption (A1). The covariance matrix $\Sigma = \D\D^T$ has entries $\Sigma_{ii} = \sigma_i^2 + \sigma_{i+1}^2$ when $1 \le i \le d$, $\Sigma_{i(i-1)} = - \sigma_i^2$ for $2 \le i \le d$, $\Sigma_{i(i+1)} = -\sigma_{i+1}^2$ for $1 \le i \le d-1$ and $\Sigma_{ij} = 0$ otherwise. In particular, (A3) is satisfied, namely $\Sigma$ is positive definite.
Moreover, $R$ is given explicitly as 
$R = I-P^T$, where $P$ is the substochastic matrix given by $P_{i (i+1)} = P_{i (i-1)} = 1/2$ for all $2 \le i \le d-1$, $P_{12} = P_{d(d-1)} = 1/2$ and $P_{ij} = 0$ if $|i-j| \ge 2$. 
 From \cite{harrison1987brownian} the process is positive recurrent and has a unique stationary distribution
if 
Assumption (A2) is satisfied, namely $\bb =-R^{-1}\mu>0$, which  is  same as the following condition.
\begin{equation}\label{atlasstab}
b_k = \sum_{i=1}^k(\delta_i - \overline{\delta})>0 \mbox{ for } 1 \le k \le d, \mbox{ where } \overline{\delta} = (d+1)^{-1}\sum_{j=1}^{d+1} \delta_j.
\end{equation}
In the special case where
\begin{equation}\label{sk}
\sigma_{i+1}^2 - \sigma_i^2 = \sigma_{2}^2 - \sigma_1^2 \text{ for all } 1 \le i \le d,
\end{equation}
 the stationary distribution is explicit and
 takes the form $\mathcal{L}(\mathbf{Y}(\infty)) = \otimes_{k=1}^d \operatorname{Exp}(2b_k\left(\sigma_k^2 + \sigma_{k+1}^2\right)^{-1})$ (see Section 5 of \cite{ichiba2011hybrid}). 
 For the general case (i.e. $\sigma_i$ are strictly positive and \eqref{atlasstab} is satisfied) explicit formulas for stationary distribution are not available, however 
 from \cite{{BL07}}, the law of $Y(t;\y)$ converges to the unique stationary distribution in (weighted) total variation distance at an exponential rate. As noted previously, this result does not provide information on parameter or dimension dependence of the rate of convergence.  
 The paper \cite{IPS} provides  explicit  rate of convergence to stationarity, 
that shows a clear parameter dependence,   under the stability condition \eqref{atlasstab} and the assumption that $\sigma_i=1$ for all $1 \le i \le d$.
In this case the stationary measure takes an explicit form and the process is reversible with respect to the stationary measure. The proofs in \cite{IPS}, which are based on Dirichlet form techniques, crucially make use of these properties.
The explicit representation of the stationary measure is  available only under the skew-symmetry condition (see \cite{harwil87}) guaranteed by \eqref{sk} and the reversibility of the process with respect to this measure is not available if the $\sigma_i$ are not all equal.
The  convergence considered in \cite{IPS} corresponds to that of time averages of bounded functionals of the state process to the corresponding stationary values in probability (see Theorem 1 of \cite{IPS}), which is considerably weaker than the $L^1$-Wasserstein distance or total variation convergence. 

From Theorem \ref{wasthm} we have the following bound on the rate of $L^1$-Wasserstein convergence of the gap process to $\mathbf{Y}(\infty)$.
Note that we do not  require reversibility or an explicit expression for the stationary measure. 

Two key quantities appearing in the rate of convergence are
\begin{equation}\label{eq:sigupplow}
a^* := \sup_{1 \le i \le d}\frac{i(d+1-i)}{b_i}, \ \ \ \sigma = \left(\sup_{1 \le i \le d} \sigma_i \right) \vee \left(\sup_{1 \le i \le d} \sigma_i^{-1} \right)
\end{equation}
where $b_i$ are defined in \eqref{atlasstab} and $\sigma_i$ is the standard deviation of the rank $i$ particle (see \eqref{rbdef}).
\begin{theorem}\label{atlasshrugged}
There exist positive constants $F_1, F_2, F_3, F_4, t_2$ such that for any $\y \in \mathbb{R}^d_+$ and any $t \ge t_2 \max\{\sigma^2a^*\|\y\|_{\infty}, 1 + \sigma^2a^{*2}\log(2d)\}$,
\begin{align*}
\mathbb{E}(\|Y(t;\y) - Y(t; \mathbf{Y}(\infty))\|_1) \le 2\left(2\|\y\|_1 + F_1\sigma^2a^*d^3\right)e^{-F_2 t/\left[d^2(1+\sigma^2a^{*2}\log(2d))\right]}\\
 + \left(4\|\y\|_1 + F_1\sigma^2a^*d^3\right)e^{-F_4t/[2\sigma^4a^{*2}(d+1)^2]} + F_3\sigma^2a^*d^{7/2}e^{-F_4t/[\sigma^4a^{*2}(d+1)^2]}.
\end{align*}
In particular, the relaxation time satisfies
\begin{align*}
t_{rel}(\y) &\le \max \left\lbrace F_2^{-1}\left[d^2(1+ \sigma^2a^{*2}\log(2d))\right]\log\left[8\left(2\|\y\|_1 + F_1\sigma^2a^*d^3\right)\right]\right.\\
&\left. \qquad \qquad\qquad + F_4^{-1}\sigma^4a^{*2}(d+1)^2\left[2\log\left[8\left(4\|\y\|_1 + F_1\sigma^2a^*d^3\right)\right] + \log (8F_3\sigma^2a^*d^{7/2})\right],\right.\\
&\left. \qquad \qquad\qquad\qquad t_2 \max\{\sigma^2a^*\|\y\|_{\infty}, 1 + \sigma^2a^{*2}\log(2d)\}\right\rbrace.
\end{align*}
\end{theorem}
\begin{proof}
Direct calculation shows that $R^{-1}$ takes the form
$$
(R^{-1})_{ij} =
\left\{
	\begin{array}{ll}
		\frac{2j(d+1-i)}{(d+1)}  & \mbox{if } j \leq i \\
		\frac{2i(d+1-j)}{(d+1)}  & \mbox{if } j > i. 
	\end{array}
\right.
$$
Therefore,
\begin{equation}\label{rank1}
\sum_{j=1}^d(R^{-1})_{ij} = (R^{-1}\one)_{i} = i(d+1-i).
\end{equation}
Using \eqref{rank1} and recalling \eqref{eq:sigupplow}, we obtain
\begin{equation}\label{rank2}
\begin{aligned}
a(\Theta) &= \sup_{1 \le i \le d}\left[\frac{\sum_{j=1}^d(R^{-1})_{ij}\sigma_j}{b_i}\right] \le \sigma \sup_{1 \le i \le d}\frac{i(d+1-i)}{b_i} = \sigma a^*,\\
b(\Theta) &= \sup_{1\le i \le d}\left[\frac{\sum_{j=1}^d(R^{-1})_{ij}\sigma_j}{\sigma_i}\right] \le \sigma^2\sup_{1 \le i \le d}i(d+1-i) \le \sigma^2\frac{(d+1)^2}{4},\\
R_2(\Theta) &= a(\Theta)^2b(\Theta) \le  \frac{\sigma^4a^{*2}(d+1)^2}{4},\\
C_1(\y,\Theta)&= 2\|\y\|_1 + a(\Theta)\sum_{i,j}(R^{-1})_{ij}\sigma_j \le 2\|\y\|_1 + \sigma^2a^*\sum_{i=1}^di(d+1-i)\\
& = 2\|\y\|_1 + \sigma^2a^*\frac{d(d+1)(d+2)}{6}.
\end{aligned}
\end{equation}
To compute $R_1(\Theta,d)$, we need to estimate $n(R)$. To do this, let $\{S^*_n\}_{n \ge 0}$ denote a simple, symmetric random walk on $\mathbb{Z}$ starting from $S^*_0 \in \{1,2,\dots,d\}$ and absorbed when it hits $0$ or $d+1$. Then for any $n \ge 0$,
$$
(P^n\one)_i = \mathbb{P}\left(S^*_n \notin \{0,d+1\} \mid S^*_0 = i\right) = \mathbb{P}\left(S^*_k \notin \{0,d+1\} \text{ for any } 1 \le k \le n \mid S^*_0 = i\right).
$$
For $j \in \{0,1,\dots, d+1\}$, define $\tau^{S^*}_j := \inf\{ k \ge 0: S^*_k = j\}$. By Chapter 10, Example 10.17 of \cite{klenke2013}, for any $i \in \{1, \ldots , d\}$, $\mathbb{E}\left(\tau^{S^*}_{d+1} \wedge \tau^{S^*}_0 \mid S^*_0 = i\right) = i(d+1-i)$. Using this observation and Markov's inequality, for all $d \in \mathbb{N}$ and $i \in \{1, \ldots , d\}$,
\begin{multline*}
\mathbb{P}\left(S^*_k \notin \{0,d+1\} \text{ for any } 1 \le k \le 2 d^2 \mid S^*_0 = i\right) \le \mathbb{P}\left(\tau^{S^*}_{d+1} \wedge \tau^{S^*}_0 > 2d^2 \mid S^*_0 = i\right)\\
\le \frac{\mathbb{E}\left(\tau^{S^*}_{d+1} \wedge \tau^{S^*}_0 \mid S^*_0 = i\right)}{2d^2}  = \frac{i(d+1 - i)}{2d^2} \le \frac{(d+1)^2}{8d^2}   \le 1/2
\end{multline*}
and consequently,
\begin{equation}\label{rank3}
n(R) \le 2d^2.
\end{equation}
Using \eqref{rank2} and \eqref{rank3}, we obtain
\begin{equation}\label{rank4}
R_1(\Theta,d) = n(R)(1+a(\Theta)^2\log(2d)) \le 2d^2(1+ \sigma^2a^{*2}\log(2d)).
\end{equation}
For $t \ge 48\sigma^2a^*\|\y\|_{\infty}$, using the bound on $a(\Theta)$ obtained in \eqref{rank2}, and noting $b(\Theta) \ge 1$ and $\|\y\|_{\infty}^* \le \sigma \|\y\|_{\infty}$,
$$
3(D_2a(\Theta)b(\Theta))^{-1}\|\y\|_{\infty}^* \le 3(D_2R_2(\Theta))^{-1}a(\Theta) \sigma\|\y\|_{\infty} \le \frac{t}{16D_2R_2(\Theta)}.
$$
Moreover, from the explicit form of $R^{-1}$, 
$
\sum_{i,j}(R^{-1})^2_{ij} \le 4d^4.
$
Using the above two bounds along with \eqref{rank2}, for $t \ge 6a^{*}(d+1)^2\|\y\|_{\infty}$,
\begin{multline}\label{rank5}
 C_2(\y,\Theta)e^{-\frac{t}{8D_2R_2(\Theta)}} = 2\|\y\|_1e^{3(D_2a(\Theta)b(\Theta))^{-1}\|\y\|_{\infty}^*}e^{-\frac{t}{8D_2R_2(\Theta)}}\\
 + a(\Theta)\left[2d(1+d)\left(\sum_{i,j}(R^{-1})^2_{ij}\right)\left(\sum_{j=1}^d\sigma_j^2\right)\right]^{1/2}e^{-\frac{t}{8D_2R_2(\Theta)}}\\
\le  2\|\y\|_1e^{-\frac{t}{16D_2R_2(\Theta)}}
 + \sigma^2a^*\left[2d^2(1+d)\left(\sum_{i,j}(R^{-1})^2_{ij}\right)\right]^{1/2}e^{-\frac{t}{8D_2R_2(\Theta)}}\\
 \le 2\|\y\|_1e^{-\frac{t}{4D_2\sigma^4a^{*2}(d+1)^2}} + \sigma^2a^*\left[8d^6(1+d)\right]^{1/2}e^{-\frac{t}{2D_2\sigma^4a^{*2}(d+1)^2}}.
\end{multline}
Take $F_1 =1, F_2 = D_1/2, F_3 =  4, F_4 = (2D_2)^{-1}$. Using the bounds obtained in \eqref{rank2}, \eqref{rank4} and \eqref{rank5} in Theorem \ref{wasthm}, for any $\y \in \mathbb{R}^d_+, t \ge \max\left\lbrace t_0 \left(1 + \sigma^2a^{*2}\log(2d)\right), 48\sigma^2a^*\|\y\|_{\infty}\right\rbrace$,
\begin{multline*}
\mathbb{E}(\|Y(t;\y) - Y(t; \mathbf{Y}(\infty))\|_1) \le C_1(\y,\Theta)\left(2e^{-\frac{D_1t}{R_1(\Theta,d)}} + e^{-\frac{t}{16D_2R_2(\Theta)}}\right)
 + C_2(\y,\Theta, D_2)e^{-\frac{t}{8D_2R_2(\Theta)}}\\
\le \left(2\|\y\|_1 + F_1\sigma^2a^*d^3\right)\left(2e^{-F_2 t/\left[d^2(1+\sigma^2a^{*2}\log(2d))\right]} + e^{-F_4t/[2\sigma^4a^{*2}(d+1)^2]}\right)\\
 + 2\|\y\|_1e^{-F_4t/[2\sigma^4a^{*2}(d+1)^2]} + F_3\sigma^2a^*d^{7/2}e^{-F_4t/[\sigma^4a^{*2}(d+1)^2]}.
\end{multline*}
This proves the first part of the theorem upon taking $t_2 = \max\left\lbrace t_0, 48\right\rbrace$. The bound on the relaxation time follows from the first part.
\end{proof}
\begin{rem}\label{sam}
The standard Atlas model \cite{Fern} is a special case of \eqref{rbdef} with $\delta_1=1$, $\delta_i = 0$ for all $i \ge 2$ and $\sigma_i=1$ for all $i$. For this model, using \eqref{atlasstab}, for any $k \ge 1$,
$$
b_k = \sum_{i=1}^k(\delta_i - \overline{\delta}) = \frac{(d+1-k)}{d+1}
$$
and
$$
a^* := \sup_{1 \le i \le d}\frac{i(d+1-i)}{b_i} = \sup_{1 \le i \le d}i(d+1) = d(d+1), \ \ \ \sigma=1.
$$
Using these in Theorem \ref{atlasshrugged}, we obtain positive constants $G_1, G_2, G_3, G_4, t_3$ such that for any $\y \in \mathbb{R}^d_+$ and any $t \ge t_3\{d^2\|\y\|_{\infty}, 1 + d^2\log(2d)\}$,
$$
\mathbb{E}(\|Y(t;\y) - Y(t; \mathbf{Y}(\infty))\|_1) \le G_1\left(\|\y\|_1 + d^5\right)e^{-G_2t/d^6\log(2d)} + G_3d^{11/2}e^{-G_4t/d^6}.
$$
In particular, the relaxation time for the standard Atlas model is $O(d^6(\log d)^2)$ as $d \rightarrow \infty$.
\end{rem}

\section{Bounding processes, small sets and return times}\label{sm}

Fix a vector $v>0$ satisfying $R^{-1}v \le \bb$  and consider the collection $\{X_v^+(\cdot; \x)\}_{\x \in \mathbb{R}^d_+} = \RBM(-v, \Sigma, I)$, where $I$ is the identity matrix, given as
\begin{equation}\label{Xplusdef}
X_v^+(t;\x) = \x + \D B(t) - vt + L^+(t),
\end{equation}
where $B$ is the same Brownian motion as used in the synchronous coupling of $\{X(\cdot; \x)\}_{\x \in \mathbb{R}^d_+}$, and $L^+$ is the local time process associated with $X_v^+$.
%
Observe that $X_v^+(\cdot;\x)$ can be written as
$$
X_v^+(t;\x) = \x + \D B(t) + \mu t + RL^*(t)
$$
where $L^*(t) = R^{-1}L^+(t) + (\bb-R^{-1}v)t$ is a non-decreasing process. By minimality of the local time process (see \cite[Appendix]{reiman1984open}), $L^*(t) \ge L(t)$ for all $t \ge 0$ implying $R^{-1}X(t;\x) \le R^{-1}X_v^+(t;\x)$ for every $t \ge 0$. Since in this section $v$ will be fixed, we abbreviate $X_v^+(\cdot;\x)$ as $X^+(\cdot;\x)$. An optimal choice of $v$ will be made later in Section \ref{optimize}. We will hereby refer to $X^+(\cdot;\x)$ as the bounding process.

We now introduce an appropriate  compact set that depends on system parameters and for which one can obtain  useful bounds on exponential moments of return times to the set. 
In order to motivate the choice of the set consider a one dimensional Brownian motion  $W_{a,b}(t) = bW(t) - at$ with variance $b^2$ and drift $-a$ (here $W$ is a standard one dimensional Brownian motion). 
Standard techniques using scale functions (see \cite[V.46]{RW2000_2}) show that for any $a>0$ and $b\in \mathbb{R}$, $ab^{-2}\sup_{t<\infty}W_{a,b}(t)$ has an Exponential distribution with mean $1/2$. This result says that the maximum of the $i$-th co-ordinate of $X^+(\cdot,\x)$ scales like $v_i\sigma_i^{-2}$. This scaling property  suggests considering the following \emph{weighted supremum norm} :
\begin{equation}\label{wnorm}
\|\x\|_{\infty,v} = \sup_{1 \le i \le d}v_i\sigma_i^{-2}x_i, \ \ \x \in \mathbb{R}^d_+.
\end{equation}
This weighted norm will play a central role in our analysis.
Also define 
\begin{equation}\label{phiddef}
\phi(v) = 2\sum_{i=1}^dv_i^2\sigma_i^{-2}/\inf_{1 \le i \le d}v_i^2\sigma_i^{-2}.
\end{equation}
Note that $\phi(v) \ge 2d$. For $A>0$, consider the compact set
$$K_A := \{\x \in \mathbb{R}_+^d: \|\x\|_{\infty,v} \le A \log \phi(v)\}$$
and define the following stopping time for the process $X^+(\cdot; \x)$:
\begin{equation}\label{tauplusde}
\tau_A^+(\x) := \inf\left\lbrace t \ge 0: X^+(t;\x) \in K_A\right\rbrace = \inf\left\lbrace t \ge 0: \|X^+(t;\x)\|_{\infty,v} \le A \log \phi(v)\right\rbrace.
\end{equation}
The following lemma gives bounds on the  exponential moments of the hitting time of the compact set $K_A$, namely $\tau_A^+(\x)$.
\begin{lemma}\label{lyaphit}
There exists $A_0>0$ such that for any $A \ge A_0$ and any $\x \in \mathbb{R}^d_+$,
$$
\mathbb{E}\left(e^{\frac{\Lambda(v)}{2A} \tau_A^+(\x)}\right) \le e^{3A^{-1}\|\x\|_{\infty,v}}
$$
where $\Lambda(v) := \inf_{1 \le i \le d}\frac{v_i^2}{\sigma_i^2}$.
\end{lemma}
\begin{proof}
Fix $A>0$ and without loss of generality assume that $\|\x\|_{\infty,v} > A \log \phi(v)$. Consider the `Lyapunov function' 
$$
V(\y) = \log\left(\sum_{i=1}^de^{g(2A^{-1} v_i\sigma_i^{-2}y_i)}\right)
$$
where $g$ is any non-negative, non-decreasing $C^2$ function defined on $\mathbb{R}_+$ such that $g'(0) = 0$, $g(u) \le u, g'(u) \le 2, g''(u) \le 9$ for all $u \ge 0$ and $g(u) = u$ for all $u \ge \log 2$. An example of such a function is $g(u) = (\log 2) h(u/\log 2)\one_{[u \le \log 2]} + u \one_{[u > \log 2]}$ where
$
h(u) = u^4 - 3u^3+3u^2.
$
The definition of the Lyapunov function is motivated by a similar function introduced in \cite{BC2016}. The main difference is that here different coordinates are weighted differently depending on system parameters.
We will prove that for sufficiently large $A$,
\begin{equation}\label{l1}
-v^T\nabla V(\y) + \frac{1}{2}\operatorname{Tr}\left(\Sigma \nabla^2V(\y)\right) + \frac{1}{2}(\nabla V(\y))^T \Sigma (\nabla V(\y)) \le -\frac{\Lambda(v)}{2A}, \ \ y \in \mathbb{R}^d_+
\end{equation}
where $\nabla$ denotes the gradient and $\nabla^2$ denotes the Hessian. By It\^{o}'s formula, this will imply that $M(t) :=  \exp\left(V(X^+(t;\x)) + \frac{\Lambda(v)}{2A} t\right)$ is a positive supermartingale
and therefore, by the optional sampling theorem, for such $A$,
$$\mathbb{E}\left(e^{\frac{\Lambda(v)}{2A} \tau^+_A(\x)}\right) \le \mathbb{E}\left(e^{V(X^+(\tau^+_A(\x);\x)) + \frac{\Lambda(v)}{2A} \tau^+_A(\x)}\right) \le e^{V(\x)}.$$
Since $\|\x\|_{\infty,v} > A \log \phi(v) > A \log d$, we  have
$$
V(\x) \le 2A^{-1}\|\x\|_{\infty,v} + \log d \le 3A^{-1} \|\x\|_{\infty,v}.
$$
Combining the two displays we   have that for $A$ that satisfy \eqref{l1} 
\begin{equation}\label{l1exp}
\mathbb{E}\left(e^{\frac{\Lambda(v)}{2A} \tau^+_A(\x)}\right) \le e^{3A^{-1} \|\x\|_{\infty,v}}.
\end{equation}
Thus in order to prove the lemma, it suffices to establish \eqref{l1} for sufficiently large $A$. Let $w_i(\y,  A) = \frac{e^{g(2A^{-1} v_i\sigma_i^{-2} y_i)}}{\sum_{k=1}^de^{g(2A^{-1} v_k\sigma_k^{-2} y_k)}}$. By similar calculations as in the proof of Lemma 4 of \cite{BC2016}, it follows that
\begin{align*}
\operatorname{Tr}\left(\Sigma \nabla^2V(\y)\right) &\le 4A^{-2} \sum_{i=1}^d \left(v_i^2\sigma_i^{-4}w_i(\y,A)\sigma_i^2\left(g''(2A^{-1} v_i\sigma_i^{-2}y_i) + g'(2A^{-1} v_i\sigma_i^{-2}y_i)^2\right)\right)\\
&\le 52A^{-2}\sum_{i=1}^dv_i^2\sigma_i^{-2}w_i(\y,A)
\end{align*}
using $g''(u) \le 9$ and $g'(u) \le 2$ for all $u \ge 0$. Moreover,
\begin{align*}
&(\nabla V(\y))^T \Sigma (\nabla V(\y))\\
 &\quad= 4A^{-2}\sum_{1 \le i,j \le d}v_i\sigma_i^{-2}w_i(\y,A)g'(2A^{-1} v_i\sigma_i^{-2}y_i) \Sigma_{ij} w_j(\y,A)v_j\sigma_j^{-2}g'(2A^{-1} v_j\sigma_j^{-2}y_j)\\
&\quad\le 4A^{-2}\left(\sum_{i=1}^dv_i\sigma_i^{-1} w_i(\y,A)g'(2A^{-1} v_i\sigma_i^{-2}y_i)\right)^2 \le 16A^{-2}\sum_{i=1}^dv_i^2\sigma_i^{-2}w_i(\y,A)
\end{align*}
where we have used $\Sigma_{ij} \le \sigma_i\sigma_j$ in the first inequality on the second line and the Cauchy-Schwarz inequality, the fact that $g'(u) \le 2$ for all $u \ge 0$,  and the fact that $\sum_{j=1}^dw_i(\y,A) = 1$ in the last inequality. From the above bounds, we obtain
\begin{equation}\label{l2}
\operatorname{Tr}\left(\Sigma \nabla^2V(\y)\right) + (\nabla V(\y))^T \Sigma (\nabla V(\y)) \le 68 A^{-2}\sum_{i=1}^dv_i^2\sigma_i^{-2}w_i(\y,A).
\end{equation}
Using the definition of $w_i$ and the monotonicity of $g$,
\begin{align*}
-v^T\nabla V(\y) &= -2A^{-1} \sum_{i=1}^dg'(2A^{-1} v_i\sigma_i^{-2} y_i) v_i^2\sigma_i^{-2}w_i(\y, A)\\
 &\le -2A^{-1} \sum_{i=1}^d v_i^2\sigma_i^{-2}w_i(\y, A) \one_{[2A^{-1} v_i\sigma_i^{-2} y_i \ge \log 2]}\\
 &= -2A^{-1} \sum_{i=1}^d v_i^2\sigma_i^{-2}w_i(\y, A) + 2A^{-1} \sum_{i=1}^d v_i^2\sigma_i^{-2}w_i(\y, A) \one_{[2A^{-1} v_i\sigma_i^{-2} y_i < \log 2]}\\
 & \le -2A^{-1} \sum_{i=1}^d v_i^2\sigma_i^{-2}w_i(\y, A) + 2A^{-1} \frac{e^{g(\log 2)}\sum_{i=1}^d v_i^2\sigma_i^{-2}}{\sum_{k=1}^de^{g(2A^{-1} v_k\sigma_k^{-2} y_k)}},
 \end{align*}
 Next, note for any $\|\y\|_{\infty,v} > A\log \phi(v)$, there is $1 \le j \le d$ such that
 $v_j\sigma_j^{-2}y_j > A \log \left(\frac{2\sum_{i=1}^dv_i^2\sigma_i^{-2}}{\inf_{1 \le i \le d}v_i^2\sigma_i^{-2}}\right)$. Hence, since $\frac{2\sum_{i=1}^dv_i^2\sigma_i^{-2}}{\inf_{1 \le i \le d}v_i^2\sigma_i^{-2}} \ge 2d \ge 2$, we obtain,
\begin{align}\label{l3}
-v^T\nabla V(\y)   &\le -2A^{-1}\sum_{i=1}^d v_i^2\sigma_i^{-2}w_i(\y, A) +  2A^{-1} \frac{2\sum_{i=1}^d v_i^2\sigma_i^{-2}}{e^{g\left(2\log \left(\frac{2\sum_{i=1}^dv_i^2\sigma_i^{-2}}{\inf_{1 \le i \le d}v_i^2\sigma_i^{-2}}\right)\right)}}\nonumber\\
  &= -2A^{-1}\sum_{i=1}^d v_i^2\sigma_i^{-2}w_i(\y, A) + A^{-1}\frac{\left(\inf_{1 \le i \le d}v_i^2\sigma_i^{-2}\right)^2}{\sum_{i=1}^dv_i^2\sigma_i^{-2}}\nonumber\\
  &\le -2A^{-1}\sum_{i=1}^d v_i^2\sigma_i^{-2}w_i(\y, A) + A^{-1}\inf_{1 \le i \le d}v_i^2\sigma_i^{-2}\nonumber\\
  &\le -2A^{-1}\sum_{i=1}^d v_i^2\sigma_i^{-2}w_i(\y, A) + A^{-1}\sum_{i=1}^d v_i^2\sigma_i^{-2}w_i(\y, A) = -A^{-1}\sum_{i=1}^d v_i^2\sigma_i^{-2}w_i(\y, A).
\end{align}
From \eqref{l2} and \eqref{l3},
\begin{multline*}
-v^T \nabla V(\y) + \frac{1}{2}\operatorname{Tr}\left(\Sigma \nabla^2V(\y)\right) + \frac{1}{2}(\nabla V(\y))^T \Sigma (\nabla V(\y))\\
 \le -A^{-1}\sum_{i=1}^d v_i^2\sigma_i^{-2}w_i(\y, A) + 34A^{-2}\sum_{i=1}^dv_i^2\sigma_i^{-2}w_i(\y,A).
\end{multline*}
Hence, for any $A \ge 68$, we obtain
\begin{multline*}
-v^T \nabla V(\y) + \frac{1}{2}\operatorname{Tr}\left(\Sigma \nabla^2V(\y)\right) + \frac{1}{2}(\nabla V(\y))^T \Sigma (\nabla V(\y))\\
 \le -\frac{1}{2A}\sum_{i=1}^d v_i^2\sigma_i^{-2}w_i(\y, A) \le -\frac{1}{2A}\inf_{1 \le i \le d}v_i^2\sigma_i^{-2} = -\frac{\Lambda(v)}{2A}
\end{multline*}
proving \eqref{l1} and hence the lemma holds with $A_0=68$. 
\end{proof}
The next lemma gives an estimate of the running maximum of a reflected Brownian motion with drift.
\begin{lemma}\label{RBMdrift}
Let $X_t = x + \sigma' B_t - \mu' t - \min\{\inf_{s \le t}\left(x + \sigma' B_s - \mu' s\right),0\}$, where $x \ge 0$, $\sigma', \mu' >0$ and $B$ is a one dimensional standard Brownian motion. Then for any $A,T>0$ and any $x \in [0, A/2]$,
$$
\mathbb{P}\left(\sup_{0 \le t \le T} X_t \ge A \right) \le e^{-\frac{\mu'^2T}{2\sigma'^2}} + (4\mu' TA^{-1} + 2)e^{-A\mu'/\sigma'^2}.
$$
\end{lemma}

\begin{proof}
Fix $A,T>0$ and $x \in [0, A/2]$. We define the following stopping times: $\tau_0=0$, and for $k \ge 0$,
\begin{align*}
\tau_{2k+1} &:= \inf\{t \ge \tau_{2k}: X_t = 0\}\\
\tau_{2k+2} &:= \inf\{t \ge \tau_{2k+1}: X_t = A/2\}.
\end{align*}
Let $\mathcal{N} := \inf\{k \ge 0: \sup_{t \in [\tau_{2k}, \tau_{2k+1}]}X_t \ge A\}$. By the strong Markov property, $\{\tau_{2k+1} - \tau_{2k}\}_{k \ge 1}$ are i.i.d, each being distributed as the hitting time of the level $-A/2$ by the process $\sigma' B_t - \mu' t$. By  \cite[Exercise 5.10]{Karatzas}, for any $\alpha>0, k \ge 1$,
$$
\mathbb{E}\left(e^{-\alpha(\tau_{2k+1} - \tau_{2k})}\right) = e^{\frac{\mu' A}{2\sigma'^2} - \frac{A}{2\sigma'}\sqrt{\frac{\mu'^2}{\sigma'^2} + 2\alpha}}.
$$
Thus, for any $n \ge 0$,
\begin{multline*}
\mathbb{P}\left(\sum_{k=0}^n(\tau_{2k+1} - \tau_{2k}) < \frac{nA}{4\mu'}\right) 
\le \mathbb{P}\left(e^{-\alpha\sum_{k=1}^n (\tau_{2k+1} - \tau_{2k})} > e^{-\alpha n A / (4\mu')} \right)\\
 \le e^{\alpha n A / (4\mu')}\mathbb{E}\left(e^{-\alpha\sum_{k=1}^n (\tau_{2k+1} - \tau_{2k})}\right) = e^{\frac{\alpha n A}{4\mu'} + \frac{n\mu' A}{2\sigma'^2} - \frac{nA}{2\sigma'}\sqrt{\frac{\mu'^2}{\sigma'^2} + 2\alpha}}.
\end{multline*}
Optimizing the above bound in $\alpha$ yields the following bound
\begin{equation}\label{d1}
\mathbb{P}\left(\sum_{k=0}^n(\tau_{2k+1} - \tau_{2k}) < \frac{nA}{4\mu'}\right) \le e^{-\frac{nA\mu'}{8\sigma'^2}}.
\end{equation}
Moreover, recalling that the scale function for the process $t \mapsto \sigma' B_t - \mu' t$ is given by $s(z) = e^{2\mu' z/\sigma'^2}$,
$$
\mathbb{P}\left(\sigma' B_t - \mu' t \text{ hits } A/2 \text{ before } -A/2\right) = \frac{1 - e^{-A\mu'/\sigma'^2}}{ e^{A\mu'/\sigma'^2} - e^{-A\mu'/\sigma'^2}} \le e^{-A\mu'/\sigma'^2}
$$
and hence, for $n \ge 1$,
\begin{equation}\label{d2}
\mathbb{P}\left(\mathcal{N} \le n\right) \le (n+1)e^{-A\mu'/\sigma'^2}.
\end{equation}
From \eqref{d1} and \eqref{d2}, for any $n \in \mathbb{N}$,
\begin{multline}\label{d3}
\mathbb{P}\left(\sup_{0 \le t \le nA/(4\mu')} X_t \ge A \mid X_0 = x\right) \le \mathbb{P}\left(\sum_{k=0}^n(\tau_{2k+1} - \tau_{2k}) < \frac{nA}{4\mu'}, \ \mathcal{N}>n\right) + \mathbb{P}\left(\mathcal{N} \le n\right)\\
\le e^{-\frac{nA\mu'}{8\sigma'^2}} + (n+1)e^{-A\mu'/\sigma'^2}.
\end{multline}
The result follows on taking $n = \lfloor4\mu' TA^{-1}\rfloor + 1$ in \eqref{d3}.
\end{proof}
Recall the quantities $\Lambda(v)$  defined in the statement of Lemma \ref{lyaphit} and $\phi(v)$  defined in \eqref{phiddef}. Define 
\begin{align*}
M(v) := \Lambda(v) + \log \phi(v), \;\;
T(v) := M(v)/\Lambda(v).
\end{align*}

The next lemma shows that  for any $C_0 \in (0,\infty)$ there are positive constants $C_1, C_2$ such that whenever $\x_1, \x_2 \in \mathbb{R}^d_+$ satisfy $\|\x_1\|_{\infty,v} \le C_0 M(v)$  and  $R^{-1}\x_2 \le R^{-1}\x_1$,  with (uniform) positive probability, all the coordinates of $X(\cdot;\x_2)$ hit zero by time $C_2 T(v)$ and  the weighted supremum norm
$\|\cdot\|_{\infty,v}$
 of $X^+(\cdot;\x_1)$ is bounded by $C_1 M(v)$ over the time
interval $[0, C_2 T(v)]$.
\begin{lemma}\label{zerohit}
For any $C_0>0$, there exists $C_1 > C_0$ and $C_2>0$ such that for any $\x_1, \x_2 \in \mathbb{R}^d_+$ satisfying $\|\x_1\|_{\infty,v} \le C_0 M(v)$ and  $R^{-1}\x_2 \le R^{-1}\x_1$,
$$
\mathbb{P}\left(\sup_{t \in [0, C_2 T(v)]} \|X^+(t;\x_1)\|_{\infty,v} \le C_1 M(v), \ \ \eta^{1}(\x_2) \le C_2 T(v)\right) \ge \frac{1}{2}.
$$
\end{lemma}
\begin{proof}
Let $\{e_i\}_{1 \le i \le d}$ denote the unit coordinate vectors in $\mathbb{R}^d$ and let $S_i := \{R^{-1}\y: \y \ge 0, \ y_i=0\}$ for $1 \le i \le d$.
Let $U(t; \x) = R^{-1}\x + R^{-1}\D B(t) + R^{-1}\mu t$. We first claim that for any $1 \le i \le d$ and $T>0$,
\begin{equation}\label{zero1}
\{U_i(t; \x) = 0 \text{ for some } 0 \le t \le T\} \subseteq \{R^{-1}X(t; \x) \in S_i \text{ for some } 0 \le t \le T\}.
\end{equation}
To see this, suppose $R^{-1}X(t; \x) \notin S_i $ for all $0 \le t \le T$. Then $X(t; \x)$ is strictly positive over $[0,T]$. Since $(R^{-1}X)_i(t; \x)= U_i(t, \x) + L_i(t)$, we have from \eqref{loctim}
that $(R^{-1}X)_i(t; \x) = U_i(t, \x)$ for all $0 \le t \le T$.
 For any $\y \in S\setminus S_i$, there exists $\z \in \mathbb{R}^d_+$ with $z_i>0$ such that $\y = R^{-1}\z$. Hence,
$$
y_i = \sum_{j = 1}^d (R^{-1})_{ij}z_j \ge (R^{-1})_{ii}z_i>0
$$
as $(R^{-1})_{ii} = (I + P^T + (P^T)^2 + \dots)_{ii} \ge 1$. Therefore, $U_i(t; \x) = (R^{-1}X)_i(t; \x) > 0$ for all $0 \le t \le T$. This proves \eqref{zero1}. 

Note that $U(\cdot; \x)$ is a Brownian motion with drift in $\mathbb{R}^d$ with covariance matrix $R^{-1} \Sigma (R^{-1})^T$ and drift vector $-b$. Write $\hat{\sigma}_i^2 := (R^{-1}\Sigma (R^{-1})^T)_{ii}$ for the variance of the $i$-th coordinate process $U_i$ of $U$. Define for each $1 \le i \le d$,
$
\tau^U_i(\x) = \inf\{t \ge 0: U_i(t;\x) =0\}.
$
Also define the vector $\hat w$ given by $\hat w_i = \sigma_i^2 v_i^{-1}$ for $1 \le i \le d$. 

For any $i$, recalling $R^{-1}v \le \bb$, note that
$$
(R^{-1}\hat w)_i = \sum_{j=1}^d(R^{-1})_{ij}\frac{\sigma_j^2}{v_j} \le \sup_{1 \le k \le d}\frac{\sigma_k^2}{v_k^2}\sum_{j=1}^d(R^{-1})_{ij}v_j \le \left(\sup_{1 \le k \le d}\frac{\sigma_k^2}{v_k^2}\right) b_i.
$$
Moreover, using $\Sigma_{jk} \le \sigma_j \sigma_k$ for all $1 \le j, k \le d$,
$$
\hat{\sigma}_i^2 = \sum_{j,k}(R^{-1})_{ij}\Sigma_{jk} (R^{-1})_{ik} \le \left(\sum_{j=1}^d(R^{-1})_{ij} \sigma_j\right)^2 \le \sup_{1 \le k \le d}\frac{\sigma_k^2}{v_k^2}\left(\sum_{j=1}^d(R^{-1})_{ij} v_j\right)^2 \le \left(\sup_{1 \le k \le d}\frac{\sigma_k^2}{v_k^2}\right) b_i^2.
$$
From the above two bounds, we conclude from the definition of $T(v)$ that for any $i$,
\begin{equation}\label{tlow}
T(v) \ge \left(\frac{(R^{-1}\hat w)_i}{b_i} \vee \frac{\hat{\sigma}_i^2}{b_i^2}\right) M(v).
\end{equation}
Fix $C'>0$ and take $\y \in \mathbb{R}^d_+$ satisfying $R^{-1}\y \le C'(R^{-1}\hat w) M(v)$. Using \eqref{tlow} and writing $N(0,1)$ for a standard normal random variable, we obtain that for any $C'' > \max\{2C', 1\}$,
\begin{multline}\label{zero2}
\mathbb{P}\left(\tau^U_i(\y) > C'' T(v)\right) \le \mathbb{P}\left((R^{-1}\y)_i + (R^{-1}\D B(C'' T(v)) + R^{-1}\mu (C'' T(v)))_i> 0\right)\\
\le \mathbb{P}\left(C'(R^{-1}\hat w)_i M(v) + \hat{\sigma}_iB_i(C'' T(v)) - b_i C'' T(v) > 0\right)\\
= \mathbb{P}\left(\hat{\sigma}_iB_i(C'' T(v))  > (b_i C'' T(v) - C'(R^{-1}\hat w)_iM(v))\right)\\
\le \mathbb{P}\left(\hat{\sigma}_iB_i(C'' T(v))  > \frac{b_iC''}{2}T(v) \right) = \mathbb{P}\left(N(0,1) > \frac{b_i}{2\hat{\sigma}_i}\sqrt{C''T(v)}\right)\\
\le \mathbb{P}\left(N(0,1) > \frac{1}{2}\sqrt{C''\log \phi(v)}\right) \le  \mathbb{P}\left(N(0,1) > \frac{1}{2}\sqrt{C''\log (2d)}\right) \le \frac{1}{(2d)^{C''/8}}
\end{multline}
where on the last line we have used  \eqref{tlow} in the first inequality and  $\phi(v) \ge 2d$ in the second inequality. 

Recall the upper bounding process $X^+$ from \eqref{Xplusdef}. Note that the $i$-th coordinate process $X_i^+$ is a one dimensional reflected Brownian motion with variance $\sigma_i^2$ and drift $-v_i$. 
Now let  $C_0>0$ be arbitrary and consider any $C' > \max\{2 C_0, 1\}$, any $C''> \max\{C',2\}$. Then, from Lemma \ref{RBMdrift}, for   any $\x \in \mathbb{R}^d_+$ satisfying $\|\x\|_{\infty,v} \le C_0M(v)$, 
\begin{equation}\label{bound1minus}
	\begin{aligned}
&\mathbb{P}\left(\sup_{t \in [0, C'' T(v)]} \|X^+(t;\x)\|_{\infty,v} > C' M(v)\right)\\
& \le \sum_{i=1}^d\mathbb{P}\left(\sup_{t \in [0, C'' T(v)]} X^+_i(t;\x) > C' \frac{\sigma_i^2}{v_i}(\Lambda(v) + \log \phi(v))\right)\\
& \le \sum_{i=1}^d\left(e^{-\frac{C''v_i^2T(v)}{2\sigma_i^2}} + \left(\frac{4C''v_i^2T(v)}{C'\sigma_i^2(\Lambda(v) + \log \phi(v))} + 2\right)e^{-C'(\Lambda(v) + \log \phi(v))}\right)\\
&\le de^{-2^{-1}C''\log \phi(v)} + \left(\frac{4C''T(v)}{C'M(v)}\sum_{i=1}^d\frac{v_i^2}{\sigma_i^2} + 2d\right)e^{-C'\log \phi(v)}.
\end{aligned}
\end{equation}
where we have used $T(v) \ge (\sigma_i^2/v_i^2) M(v) \ge (\sigma_i^2/v_i^2)\log \phi(v)$ in the last step. From the definition of $\phi(v)$ and $T(v)$ respectively,
$$
\sum_{i=1}^d\frac{v_i^2}{\sigma_i^2}  = \frac{\phi(v)}{2} \Lambda(v), \ \text{ and } \frac{T(v)}{M(v)} = 1/\Lambda(v).
$$
Using these observations in \eqref{bound1minus}, we obtain
\begin{multline}\label{bound1}
\mathbb{P}\left(\sup_{t \in [0, C'' T(v)]} \|X^+(t;\x)\|_{\infty,v} > C' M(v)\right) \le de^{-C''\log \phi(v)/2} + \left(\frac{2C''\phi(v)}{C'} + 2d\right)e^{-C'\log \phi(v)}\\
\le \frac{1}{(2d)^{\frac{C''}{2} - 1}} + \left(\frac{2C''}{C'} + 1\right)\frac{1}{(2d)^{C'-1}}
\end{multline}
where once more we have used $\phi(v) \ge 2d$ to obtain the last bound.

Note that for any $C', C''>0$, for any $\x_1, \x_2 \in \mathbb{R}^d_+$ satisfying $\|\x_1\|_{\infty,v} \le C_0 M(v)$ and  $R^{-1}\x_2 \le R^{-1}\x_1$,
\begin{multline}\label{brk1}
\mathbb{P}\left(\sup_{t \in [0, C'' T(v)]} \|X^+(t;\x_1)\|_{\infty,v} > C' M(v) \ \text{ or } \ \eta^{1}(\x_2) > C'' T(v)\right)\\
= \mathbb{P}\left(\sup_{t \in [0, C'' T(v)]} \|X^+(t;\x_1)\|_{\infty,v} > C' M(v)\right)\\
+ \mathbb{P}\left(\eta^{1}(\x_2) > C'' T(v), \  \sup_{t \in [0, C'' T(v)]} \|X^+(t;\x_1)\|_{\infty,v} \le C' M(v)\right).
\end{multline}
Note that for $\z_1, \z_2 \in \mathbb{R}_+^d$, if $\|\z_1\|_{\infty,v} \le C M(v)$ for some $C >0$, and  $R^{-1}\z_2 \le R^{-1}\z_1$,  then $R^{-1}\z_2 \le C (R^{-1}\hat w) M(v)$. Using \eqref{zero1} and \eqref{zero2}, choosing any  $C'' > 1 + \max\{2C', 1\}$, we obtain by the Markov property applied at time $1$,
\begin{multline}\label{brk3}
\mathbb{P}\left(\eta^{1}(\x_2) > C'' T(v), \  \sup_{t \in [0, C'' T(v)]} \|X^+(t;\x_1)\|_{\infty,v} \le C' M(v)\right)\\
 \le \sum_{i = 1}^d\sup_{\y \in \mathbb{R}^d_+: R^{-1}\y \le C'(R^{-1}w) M(v)}\mathbb{P}\left(\tau^U_i(\y) > (C''-1) T(v)\right)
\le \frac{d}{(2d)^{(C''-1)/8}} = \frac{1}{2(2d)^{(C''-9)/8}}.
\end{multline}
Using the estimates \eqref{bound1} and \eqref{brk3} in \eqref{brk1}, we obtain for $C' > \max\{2 C_0, 8\}$,  $C''  = 2 + \max\{2C',33\}$, and any $\x_1, \x_2 \in \mathbb{R}^d_+$ satisfying $\|\x_1\|_{\infty,v} \le C_0 M(v)$ and  $R^{-1}\x_2 \le R^{-1}\x_1$,
\begin{multline*}
\mathbb{P}\left(\sup_{t \in [0, C'' T(v)]} \|X^+(t;\x_1)\|_{\infty,v} > C' M(v) \ \text{ or } \ \eta^{1}(\x_2) > C'' T(v)\right)\\
\le \frac{1}{(2d)^{\frac{C''}{2} - 1}} + \left(\frac{2C''}{C'} + 1\right)\frac{1}{(2d)^{C'-1}} + \frac{1}{2(2d)^{(C''-9)/8}} < \frac{1}{2}.
\end{multline*}
The lemma follows on taking $C_1 = C', C_2 = C''$.
\end{proof}
\begin{rem}\label{ss}
Recall the quantity  $A_0$ from Lemma \ref{lyaphit} and consider $C_0\ge A_0$. Let  $C_1$ be as in Lemma \ref{zerohit} associated with this $C_0$. Then the set 
$S := \{\y: \|\y\|_{\infty , v} \le C_1M(v)\}$ plays a role similar to that of a `small set' in the theory developed in \cite{meyn2012markov}, in the following sense. For any $\x \ge 0$, (i) by Lemma \ref{lyaphit}, we have tight control over return times of the bounding process $X^+(\cdot;\x)$ to the set $S'=\{\y: \|\y\|_{\infty, v} \le C_0M(v)\} \subset S$, and (ii) by Lemma \ref{zerohit}, given that the bounding process $X^+(t;\x)$ lies in $S'$ for some $t \ge 0$, then with probability at least a half, all the co-ordinates of $X(\cdot;\x)$ hit zero at least once in the time interval $[t, t + T(v)]$ without $X^+(\cdot;\x)$ leaving $S$. This, in view of  Lemma \ref{contract}, says that  $\|X(\cdot;\x) - X(\cdot;\0)\|_1$ is reduced by a factor $2^{-1/n(R)}$ over this time interval.
\end{rem}
\section{Excursions from the small set}\label{excth}
In the following lemma, we combine the estimates from Sections \ref{con} and \ref{sm} to decompose the path of $X^+(\cdot;\x)$ into excursions from the small set (described in Remark \ref{ss}) and quantify the rate of decay of $\|X(t;\x) - X(t;\0)\|_1$ as $t$ increases. 
\begin{lemma}\label{wassest}
For any $A \ge A_0$, where $A_0$ is the constant appearing in Lemma \ref{lyaphit}, there exist positive constants $t_0, D_1$ such that for any $\x \in \mathbb{R}^d_+$,
any $v>0$ with $R^{-1}v \le \bb$, and any
$t \ge t_0 T(v)$,
\begin{equation*}
\mathbb{E}(\|X(t;\x) - X(t; \0)\|_1) \le 2\|\x\|_1\left(2e^{-\frac{D_1t}{n(R)T(v)}} + e^{-\frac{\Lambda(v)t}{16A}}\right) + 2\|\x\|_1e^{3A^{-1}\|\x\|_{\infty,v}}e^{-\frac{\Lambda(v)t}{8A}}.
\end{equation*} 
\end{lemma}
\begin{proof}
Fix $A\ge A_0$ and consider constants $C_1, C_2$ from Lemma \ref{zerohit} that are associated with $C_0=A$.
We also consider the following stopping times. Let $\tau_0 = \inf\{t \ge 0: \|X^+(t;\x)\|_{\infty,v} \le C_1 M(v)\}$. For $k \ge 0$, having defined the stopping times $\tau_0, \dots, \tau_{2k}$, define
\begin{align*}
\tau_{2k+1} &:= \inf\{t \ge \tau_{2k}: \|X^+(t;\x)\|_{\infty,v} \le C_0 M(v)\},\\
\tau_{2k+2} &:= \inf\{t \ge \tau_{2k+1}: \|X^+(t;\x)\|_{\infty,v} = C_1M(v)\} \wedge (\tau_{2k+1} + C_2T(v)).
\end{align*}
Define
$
\mathbf{N}_t := \inf\{k \ge 0: \tau_{2k} \le t\}.
$
For $k \ge 0$, define the event
\begin{multline*}
E_k :=\{\tau_{2k+2} \ge \tau_{2k+1} +1, \ \text{ and all the co-ordinates of } \{X(t;\x)\}_{t \ge 0} \text{ hit zero }\\
\text{ in the time interval } [\tau_{2k+1} +1, \tau_{2k+2}]\}.
\end{multline*}
On the event $E_k$, all the coordinates of $X(\cdot;\x)$ hit zero in the time interval $[\eta^{\mathcal{N}(\tau_{2k}; \x)} + 1, \tau_{2k+2}]$ as it contains the interval $[\tau_{2k+1} + 1, \tau_{2k+2}]$. Consequently, $\mathcal{N}(\tau_{2k+2}; \x) - \mathcal{N}(\tau_{2k}; \x) \ge 1$. Thus, for any $k \ge 0$,
\begin{equation*}
\mathcal{N}(\tau_{2k+2}; \x) - \mathcal{N}(\tau_{k}; \x) \ge \one_{E_k}.
\end{equation*}
Hence,
\begin{equation}\label{was1}
\mathcal{N}(t; \x) \ge \sum_{k=0}^{\mathbf{N}_t - 1} \one_{E_k}.
\end{equation}
Let $\mathcal{F}_t := \sigma\{B(s): 0 \le s \le t\}$ be the filtration generated by the Brownian motion.
For $k \ge 0$, let 
$$
M_n := \sum_{k=0}^{n-1}\left(\one_{E_k} - \mathbb{E}\left(\one_{E_k} \mid \mathcal{F}_{\tau_{2k}}\right)\right).
$$
Then $(M_n,\mathcal{F}_{\tau_{2n}})_{n \ge 1}$ is a martingale with increments bounded by $1$.
By Lemma \ref{zerohit}, for every $k \ge 0$,
$
\mathbb{E}\left(\one_{E_k} \mid \mathcal{F}_{\tau_{2k}}\right) \ge 1/2.
$
Thus, for any $\delta' >0$, using the Azuma-Hoeffding inequality with $t \ge 4T(v)/\delta'$,
\begin{multline}\label{was1.5}
\mathbb{P}\left(\sum_{k=0}^{\mathbf{N}_t - 1} \one_{E_k} < \delta' t/(4 T(v)), \  \mathbf{N}_t \ge \frac{\delta' t}{T(v)}\right) \le \mathbb{P}\left(\sum_{k=0}^{\lfloor\frac{\delta' t}{T(v)}\rfloor - 1} \one_{E_k} < \delta' t/(4 T(v))\right)\\
=  \mathbb{P}\left(M_{\lfloor\frac{\delta' t}{T(v)}\rfloor} < \frac{\delta' t}{4 T(v)} -  \sum_{k=0}^{\lfloor\frac{\delta' t}{T(v)}\rfloor - 1}\mathbb{E}\left(\one_{E_k} \mid \mathcal{F}_{\tau_{2k}}\right)\right) \le \mathbb{P}\left(M_{\lfloor\frac{\delta' t}{T(v)}\rfloor} < \frac{\delta' t}{4 T(v)} -  \frac{1}{2}\lfloor\frac{\delta' t}{T(v)}\rfloor\right)\\
\le \mathbb{P}\left(M_{\lfloor\frac{\delta' t}{T(v)}\rfloor} < -\frac{\delta' t}{8 T(v)}\right) \le e^{-\delta' t/(128 T(v))}.
\end{multline}
Note that for any $\delta' \in (0, C_2^{-1}/2]$,
\begin{multline}\label{was2}
\mathbb{P}\left(\mathbf{N}_t < \delta' t/T(v)\right) \le \mathbb{P}\left(\tau_0 + \sum_{k=0}^{ \lfloor\frac{\delta' t}{T(v)}\rfloor} (\tau_{2k+2} - \tau_{2k+1}) + \sum_{k=0}^{ \lfloor\frac{\delta' t}{T(v)}\rfloor} (\tau_{2k+1} - \tau_{2k}) > t\right)\\
\le \mathbb{P}\left(\tau_0 + \sum_{k=0}^{ \lfloor\frac{\delta' t}{T(v)}\rfloor} (\tau_{2k+1} - \tau_{2k}) > t/2\right),
\end{multline}
where the last inequality follows because $\tau_{2k+2} - \tau_{2k+1} \le C_2 T(v)$ and hence, as $\delta' \in (0, C_2^{-1}/2]$,
\begin{equation}\label{ext}
\sum_{k=0}^{ \lfloor\frac{\delta' t}{T(v)}\rfloor} (\tau_{2k+2} - \tau_{2k+1}) \le t/2.
\end{equation}
Since $C_0=A$ and $M(v) > \log \phi(v)$, for any $k \ge 0$, conditionally on $\mathcal{F}_{\tau_{2k}}$, $\tau_{2k+1} - \tau_{2k}$ is stochastically dominated by $\tau_{A}^+(X^+(\tau_{2k};\x))$ where $\tau_{A}^+(\cdot)$ is defined in \eqref{tauplusde}. 

For any $n \ge 0$ $s>0$, using Lemma \ref{lyaphit}, we obtain
\begin{align*}
\mathbb{P}\left(\sum_{k=0}^{n} (\tau_{2k+1} - \tau_{2k}) > s \right) &\le e^{-\frac{\Lambda(v) s}{2A}}\mathbb{E}\left(e^{\frac{\Lambda(v)}{2A}\left(\sum_{k=0}^{n} (\tau_{2k+1} - \tau_{2k})\right)}\right)\\
 &= e^{-\frac{\Lambda(v) s}{2A}}\mathbb{E}\left(e^{\frac{\Lambda(v)}{2A}\left(\sum_{k=0}^{n-1} (\tau_{2k+1} - \tau_{2k})\right)}\mathbb{E}\left(e^{\frac{\Lambda(v)}{2A} (\tau_{2n+1} - \tau_{2n})}\ \Big| \ \mathcal{F}_{\tau_{2n}}\right)\right)\\
 &\le e^{-\frac{\Lambda(v) s}{2A}}\mathbb{E}\left(e^{\frac{\Lambda(v)}{2A}\left(\sum_{k=0}^{n-1} (\tau_{2k+1} - \tau_{2k})\right)}e^{3A^{-1}\|X^+(\tau_{2n};\x)\|_{\infty,v}}\right)\\
 &\le e^{-\frac{\Lambda(v) s}{2A}}e^{3A^{-1}C_1M(v)}\mathbb{E}\left(e^{\frac{\Lambda(v)}{2A}\left(\sum_{k=0}^{n-1} (\tau_{2k+1} - \tau_{2k})\right)}\right)
\end{align*}
where we have used $\|X^+(\tau_{2n};\x)\|_{\infty,v} \le C_1M(v)$ by definition of $\tau_{2n}$, and we take $\sum_{k=0}^{n-1} (\tau_{2k+1} - \tau_{2k}) = 0$ when $n=0$. Iteratively using the same argument, we obtain
\begin{equation}\label{was3x}
\mathbb{P}\left(\sum_{k=0}^{n} (\tau_{2k+1} - \tau_{2k}) > s \right) \le e^{-\frac{\Lambda(v) s}{2A}}e^{3(n+1)A^{-1}C_1M(v)}.
\end{equation}
From \eqref{was3x}, for any positive $\delta' \le \min\{C_2^{-1}/2, (64C_1)^{-1}\}$ and  $t \ge 3T(v)/\delta'$, taking $n= \lfloor\frac{\delta' t}{T(v)}\rfloor$ and $s=t/4$,
\begin{multline}\label{was3}
\mathbb{P}\left(\sum_{k=0}^{\lfloor\frac{\delta' t}{T(v)}\rfloor} (\tau_{2k+1} - \tau_{2k}) > t/4\right) \le e^{-\frac{\Lambda(v) t}{8A}}e^{3\delta'tA^{-1}C_1M(v)/T(v)}e^{3M(v)A^{-1}C_1}\\
\le e^{-\frac{\Lambda(v) t}{8A}}e^{4\delta'tA^{-1}C_1M(v)/T(v)}
= e^{-\frac{\Lambda(v) t}{8A}}e^{4\delta'tA^{-1}C_1\Lambda(v)} \le e^{-\frac{\Lambda(v) t}{16A}}
\end{multline}
where we have used $T(v) = M(v)/\Lambda(v)$ in the equality above. Moreover, as $\tau_0 \le \tau^+_A(\x)$, by Lemma \ref{lyaphit},
\begin{equation}\label{was4}
\mathbb{P}\left(\tau_0 > t/4\right) \le e^{-\frac{\Lambda(v)t}{8A}}\mathbb{E}\left(e^{\frac{\Lambda(v)}{2A} \tau^+_A(\x)}\right) \le e^{-\frac{\Lambda(v)t}{8A}}e^{3A^{-1}\|\x\|_{\infty,v}}.
\end{equation}
Using \eqref{ext}, \eqref{was3} and \eqref{was4} in \eqref{was2}, for any $\delta'$ and $t$ as above,
\begin{multline}\label{was5}
\mathbb{P}\left(\mathbf{N}_t < \delta' t/T(v)\right) \le \mathbb{P}\left(\sum_{k=0}^{\frac{1}{2}\lfloor \delta' t/T(v)\rfloor} (\tau_{2k+1} - \tau_{2k}) > t/4\right) + \mathbb{P}\left(\tau_0 > t/4\right)\\
\le e^{-\frac{\Lambda(v)t}{16A}} + e^{-\frac{\Lambda(v)t}{8A}}e^{3A^{-1}\|\x\|_{\infty,v}}.
\end{multline}
From \eqref{was1.5} and \eqref{was5}, for positive $\delta' \le \min\{C_2^{-1}/2, (64C_1)^{-1}\}$ and $t \ge 4 T(v)/\delta'$,
\begin{multline}\label{was5.5}
\mathbb{P}\left(\sum_{k=0}^{\mathbf{N}_t - 1} \one_{E_k} < \delta' t/(4T(v))\right)\\
 \le \mathbb{P}\left(\sum_{k=0}^{\mathbf{N}_t - 1} \one_{E_k} < \delta' t/(4 T(v)), \  \mathbf{N}_t \ge \frac{\delta' t}{T(v)}\right) + \mathbb{P}\left(\mathbf{N}_t < \delta' t/T(v)\right)\\
 \le e^{-\delta' t/(128 T(v))} + e^{-\frac{\Lambda(v)t}{16A}} + e^{-\frac{\Lambda(v)t}{8A}}e^{3A^{-1}\|\x\|_{\infty,v}}.
\end{multline}
Now, using Lemma \ref{contract}, \eqref{was1} and \eqref{was5.5}, for positive $\delta' \le \min\{C_2^{-1}/2, (64C_1)^{-1}\}$ and $t \ge 4 T(v)/\delta'$,
\begin{multline}\label{concal}
\mathbb{E}(\|X(t;\x) - X(t; \0)\|_1) \le 2\|\x\|_1 \mathbb{E}\left(2^{-\mathcal{N}(t;\x)/n(R)}\right) \le 2\|\x\|_1 \mathbb{E}\left(2^{-\frac{1}{n(R)}\sum_{k=0}^{\mathbf{N}_t - 1} \one_{E_k}}\right)\\
\le 2\|\x\|_1 \mathbb{P}\left(\sum_{k=0}^{\mathbf{N}_t - 1} \one_{E_k} < \delta' t/(4T(v))\right) + 2\|\x\|_{1}2^{-\frac{\delta't}{4n(R)T(v)}}\\
\le 2\|\x\|_1\left(e^{-\delta' t/(128 T(v))} + e^{-\frac{\Lambda(v)t}{16A}} + e^{-\frac{\Lambda(v)t}{8A}}e^{3A^{-1}\|\x\|_{\infty,v}} + 2^{-\frac{\delta't}{4n(R)T(v)}} \right)\\
\le 2\|\x\|_1\left(2e^{-\frac{\delta't}{128n(R)T(v)}} + e^{-\frac{\Lambda(v)t}{16A}}\right) + 2\|\x\|_1e^{3A^{-1}\|\x\|_{\infty,v}}e^{-\frac{\Lambda(v)t}{8A}}.
\end{multline}
This proves the lemma with
\begin{equation} \label{eq:rev1} t_0 = 4/\delta', \; D_1 = \delta'/128.
	\end{equation}
\end{proof}
\section{Main result: optimizing over $v$}\label{optimize}
In this section, we state and prove our main theorem. This will involve optimizing the bound obtained in Lemma \ref{wassest} over all possible choices of $v$ along with making an appropriate choice of $A$.

{\em Proof of Theorem \ref{wasthm}.}
Fix any $A \ge A_0$ whose value will be appropriately chosen later. Recall from Lemma \ref{wassest} that the quantities $\frac{D_1}{n(R)T(v)}$ and $\frac{\Lambda(v)}{A}$ govern the rate of decay of $\mathbb{E}(\|X(t;\x) - X(t; \0)\|_1)$ for any $\x \in \mathbb{R}^d_+$.
To obtain the result in the theorem, we first obtain a value $v(\Theta)$ of $v$ which simultaneously maximizes $\frac{D_1}{n(R)T(v)}$ and $\frac{\Lambda(v)}{A}$ over all vectors $v>0$ satisfying $R^{-1}v \le \bb$. For any such $v$, define the vector $s(v) >0$ by
$$
(s(v))_i = \sigma_i\inf_{1 \le k \le d}\frac{v_k}{\sigma_k}, \ \ \ 1 \le i \le d. 
$$
Then, from the definition of $\Lambda$ we see that $\Lambda(s(v)) = \Lambda(v)$. Moreover, $\phi(s(v)) = 2d \le \phi(v)$ from definition of $\phi$. Therefore,
$$
T(s(v)) = \frac{\Lambda(s(v)) + \log (2d)}{\Lambda(s(v))} \le \frac{\Lambda(v) + \log(\phi(v))}{\Lambda(v)} = T(v).
$$
Thus, for maximizing the rate, it suffices to restrict attention to vectors $v$ of the form $v_i = v_*\sigma_i$ for $1 \le i \le d$. From the constraint $R^{-1}v \le \bb$, we obtain
$$
b_i \ge (R^{-1}v)_i = v_*\sum_{j=1}^d(R^{-1})_{ij}\sigma_j, \mbox{ for all } i,
$$
and hence,
$$
v_* \le \inf_{1\le i \le d}\left[\frac{b_i}{\sum_{j=1}^d(R^{-1})_{ij}\sigma_j}\right] = \frac{1}{a(\Theta)}.
$$
From this observation, it follows that for any such vector $v$, $\Lambda(v) = v_*^2 \le (a(\Theta))^{-2}$ and $T(v) \ge 1 + (a(\Theta))^2\log(2d)$. The vector $\widetilde{v}$ given by $\widetilde{v}_i := (a(\Theta))^{-1} \sigma_i$ for each $i$ satisfies $R^{-1}\widetilde{v} \le \bb$ and hence, simultaneously maximizes $\frac{D_1}{n(R)T(v)}$ and $\frac{\Lambda(v)}{A}$. 
Also note that $T(\widetilde{v}) = 1 + (a(\Theta))^2\log(2d)$ and $\Lambda(\widetilde{v}) = (a(\Theta))^{-2}$.

From Lemma \ref{wassest} with $\widetilde{v}$ in place of $v$, we obtain for any $\x \in \mathbb{R}^d_+$ and $t \ge t_0 \left(1 + (a(\Theta))^2\log(2d)\right)$,
\begin{equation}\label{optx}
\mathbb{E}(\|X(t;\x) - X(t; \0)\|_1) \le 2\|\x\|_1\left(2e^{-\frac{D_1t}{R_1(\Theta,d)}} + e^{-\frac{t}{16A(a(\Theta))^2}}\right)
+ 2\|\x\|_1e^{3A^{-1}\|\x\|_{\infty,\widetilde{v}}}e^{-\frac{t}{8A(a(\Theta))^2}}.
\end{equation}
Consider the dominating process $X^+(\cdot;\x) = X_v^+(\cdot;\x)$ with $v = \widetilde{v}$. Since $R^{-1}X(t;\x) \le R^{-1}X^+(t;\x)$, for each $i$,
\begin{align*}
X_i(t;\x) &\le (R^{-1}X(t;\x))_i \le (R^{-1}X^+(t;\x))_i = \sum_{j=1}^d(R^{-1})_{ij}X^+_j(t;\x)\\
&\le \left(\sup_{1 \le k \le d} \frac{X^+_k(t;\x)}{\sigma_k}\right)\sum_{j=1}^d(R^{-1})_{ij}\sigma_j
\end{align*}
and hence,
$$
\|X(t;\x)\|_{\infty,\widetilde{v}} \le b(\Theta)\|X^+(t;\x)\|_{\infty,\widetilde{v}}, \ \ t \ge 0.
$$
Moreover,
$$
\|X(t;\x)\|_{1} \le \sum_{i,j}(R^{-1})_{ij}X_j(t;\x) \le \sum_{i,j}(R^{-1})_{ij}X^+_j(t;\x), \ \ t \ge 0.
$$
Denote by $X(\infty)$ and $X^+(\infty)$ the random vectors sampled from the stationary distribution of $X(\cdot;\x)$ and $X^+(\cdot;\x)$ respectively. By \cite{BL07}, the laws of $X(t;\x)$ and $X^+(t;\x)$ converge in total variation to those of $X(\infty)$ and $X^+(\infty)$ respectively.
Consequently, $\|X(\infty)\|_{1}$ is stochastically dominated by $\sum_{i,j}(R^{-1})_{ij}X^+_j(\infty)$ and $\|X(\infty)\|_{\infty,\widetilde{v}}$ is stochastically dominated by $b(\Theta)\|X^+(\infty)\|_{\infty,\widetilde{v}}$. As $X_i^+(\infty)$ is the stationary distribution of a one-dimensional reflected Brownian motion with drift $-\widetilde{v}_i$ and variance $\sigma_i^2$, it is a standard fact that $X_i^+(\infty)$ follows an exponential distribution with mean $\sigma_i^2/(2\widetilde{v}_i) = \sigma_i a(\Theta)/2$. This implies
\begin{align*}
\mathbb{E}\left(\|X(\infty)\|_{1}\right) \le \mathbb{E}\left(\sum_{i,j}(R^{-1})_{ij}X^+_j(\infty)\right) &=\frac{a(\Theta)}{2}\sum_{i,j}(R^{-1})_{ij}\sigma_j,\\
\mathbb{E}\left(\|X(\infty)\|_{1}\right)^2 \le \mathbb{E}\left(\sum_{i,j}(R^{-1})_{ij}X^+_j(\infty)\right)^2 &\le \frac{d(a(\Theta))^2}{2}\left(\sum_{i,j}(R^{-1})^2_{ij}\right)\left(\sum_{j=1}^d\sigma_j^2\right)
\end{align*}
and
$$
\mathbb{P}\left(\|X^+(\infty)\|_{\infty,\widetilde{v}} > t\right) \le \sum_{i=1}^d \mathbb{P}\left(X_i^+(\infty) > a(\Theta)\sigma_it\right) = de^{-2t}.
$$
Using the above estimates we have for any  $D_2 \ge 9$,
\begin{multline*}
\mathbb{E}\left(e^{9(D_2b(\Theta))^{-1}\|X(\infty)\|_{\infty,\widetilde{v}}}\right) \le \mathbb{E}\left(e^{9D_2^{-1}\|X^+(\infty)\|_{\infty,\widetilde{v}}}\right) \le \int_0^{\infty}\mathbb{P}\left(\|X^+(\infty)\|_{\infty,\widetilde{v}} > D_2\log t/9\right)\operatorname{dt}\\
 \le  1 + d\int_1^{\infty}e^{-2D_2\log t/9}\operatorname{dt} \le 1 + d\int_1^{\infty}t^{-2}\operatorname{dt} = 1 + d.
\end{multline*}
Now fix $D_2 = \max\{A_0, 9\}$. Using the above estimates and \eqref{optx} with $A = D_2b(\Theta)$ (noting that $b(\Theta) \ge 1$), we obtain for any $\x \in \mathbb{R}^d_+$ and $t \ge t_0 \left(1 + (a(\Theta))^2\log(2d)\right)$,
\begin{align*}
&\mathbb{E}(\|X(t;\x) - X(t; \mathbf{X}(\infty))\|_1) \le \mathbb{E}(\|X(t;\x) - X(t; \0)\|_1) + \mathbb{E}(\|X(t; \mathbf{X}(\infty)) - X(t; \0)\|_1)\\
&\le  2\|\x\|_1\left(2e^{-\frac{D_1t}{R_1(\Theta,d)}} + e^{-\frac{t}{16D_2R_2(\Theta)}}\right)
+ 2\|\x\|_1e^{3(D_2b(\Theta))^{-1}\|\x\|_{\infty,\widetilde{v}}}e^{-\frac{t}{8D_2R_2(\Theta)}}\\
& \qquad + 2\mathbb{E}\left(\|X(\infty)\|_1\right)\left(2e^{-\frac{D_1t}{R_1(\Theta,d)}} + e^{-\frac{t}{16D_2R_2(\Theta)}}\right) + 2\mathbb{E}\left(\|X(\infty)\|_1e^{3(D_2b(\Theta))^{-1}\|X(\infty)\|_{\infty,\widetilde{v}}}\right)e^{-\frac{t}{8D_2R_2(\Theta)}}\\
&\le 2\|\x\|_1\left(2e^{-\frac{D_1t}{R_1(\Theta,d)}} + e^{-\frac{t}{16D_2R_2(\Theta)}}\right)
+ 2\|\x\|_1e^{3(D_2b(\Theta))^{-1}\|\x\|_{\infty,\widetilde{v}}}e^{-\frac{t}{8D_2R_2(\Theta)}}\\
& \qquad \qquad + 2\mathbb{E}\left(\|X(\infty)\|_1\right)\left(2e^{-\frac{D_1t}{R_1(\Theta,d)}} + e^{-\frac{t}{16D_2R_2(\Theta)}}\right)\\
&\qquad \qquad \qquad + 2\left(\mathbb{E}\left(\|X(\infty)\|_1^2\right)\right)^{1/2}\left(\mathbb{E}\left(e^{9(D_2b(\Theta))^{-1}\|X(\infty)\|_{\infty,\widetilde{v}}}\right)\right)^{1/2}e^{-\frac{t}{8D_2R_2(\Theta)}}\\ 
&\le 2\|\x\|_1\left(2e^{-\frac{D_1t}{R_1(\Theta,d)}} + e^{-\frac{t}{16D_2R_2(\Theta)}}\right)
+ 2\|\x\|_1e^{3(D_2b(\Theta))^{-1}\|\x\|_{\infty,\widetilde{v}}}e^{-\frac{t}{8D_2R_2(\Theta)}}\\
& \qquad \qquad + a(\Theta)\sum_{i,j}(R^{-1})_{ij}\sigma_j\left(2e^{-\frac{D_1t}{R_1(\Theta,d)}} + e^{-\frac{t}{16D_2R_2(\Theta)}}\right)\\
&\qquad \qquad \qquad + a(\Theta)\left[2d(1+d)\left(\sum_{i,j}(R^{-1})^2_{ij}\right)\left(\sum_{j=1}^d\sigma_j^2\right)\right]^{1/2}e^{-\frac{t}{8D_2R_2(\Theta)}}
\end{align*}
which proves the Wasserstein bound in the theorem upon noting that $\|\x\|_{\infty,\widetilde{v}} = (a(\Theta))^{-1}\|\x\|_{\infty}^*$. This, in turn, implies the stated bound on the relaxation time.
\hfill \qed

\begin{rem}\label{bc}
To obtain the better bound displayed in Remark \ref{betterconst}, note that we can replace the bound in Lemma \ref{wassest} by
$$
\mathbb{E}(\|X(t;\x) - X(t; \0)\|_1) \le 2\|\x\|_1\left(e^{-\frac{\delta' t}{128 T(v)}} + e^{-\frac{\Lambda(v)t}{16A}} + e^{-\frac{\Lambda(v)t}{8A}}e^{3A^{-1}\|\x\|_{\infty,v}} + 2^{-\frac{\delta't}{4n(R)T(v)}} \right)
$$
which follows from the second-to-last inequality in the calculation \eqref{concal}. The vector $\widetilde{v}$ in the proof of Theorem \ref{wasthm} still optimizes the above bound over all $v$ and leads to the bound displayed in Remark \ref{betterconst}.
\end{rem}

$\,$\\
{\bf Acknowledgements.}  SB was partially supported by a Junior Faculty Development Award made by UNC, Chapel Hill. 
Research of AB was partially supported by the National Science
Foundation (DMS- 1814894, DMS-1853968). We thank Kavita Ramanan and Andrey Sarantsev for  helpful discussions. We also thank two anonymous referees whose comments and suggestions helped to improve the article.

\bibliographystyle{plain}

\begin{thebibliography}{10}

\bibitem{atar2001positive}
Rami Atar, Amarjit Budhiraja, and Paul Dupuis.
\newblock On positive recurrence of constrained diffusion processes.
\newblock {\em Annals of probability}, pages 979--1000, 2001.

\bibitem{BFK}
Adrian~D Banner, Robert Fernholz, and Ioannis Karatzas.
\newblock Atlas models of equity markets.
\newblock {\em The Annals of Applied Probability}, 15(4):2296--2330, 2005.

\bibitem{BC2016}
Jose Blanchet and Xinyun Chen.
\newblock Rates of convergence to stationarity for multidimensional {R}{B}{M}.
\newblock {\em arXiv preprint arXiv:1601.04111}, 2016.

\bibitem{bolley2010trend}
Fran{\c{c}}ois Bolley, Arnaud Guillin, and Florent Malrieu.
\newblock Trend to equilibrium and particle approximation for a weakly
  selfconsistent {V}lasov-{F}okker-{P}lanck equation.
\newblock {\em ESAIM: Mathematical Modelling and Numerical Analysis},
  44(5):867--884, 2010.

\bibitem{bramson2001heavy}
Maury Bramson and Jim~G Dai.
\newblock Heavy traffic limits for some queueing networks.
\newblock {\em Annals of Applied Probability}, pages 49--90, 2001.

\bibitem{buddupsimp}
Amarjit Budhiraja and Paul Dupuis.
\newblock Simple necessary and sufficient conditions for the stability of
  constrained processes.
\newblock {\em SIAM Journal on Applied Mathematics}, 59(5):1686--1700, 1999.

\bibitem{BL07}
Amarjit Budhiraja and Chihoon Lee.
\newblock {Long time asymptotics for constrained diffusions in polyhedral
  domains}.
\newblock {\em Stoch. Proc. Appl.}, 117(8):1014--1036, 2007.

\bibitem{dai1995positive}
Jim~G Dai.
\newblock On positive {H}arris recurrence of multiclass queueing networks: {A}
  unified approach via fluid limit models.
\newblock {\em The Annals of Applied Probability}, pages 49--77, 1995.

\bibitem{dupuis1994lyapunov}
Paul Dupuis and Ruth~J Williams.
\newblock Lyapunov functions for semimartingale reflecting {B}rownian motions.
\newblock {\em The Annals of Probability}, pages 680--702, 1994.

\bibitem{Eberle2015}
Andreas Eberle.
\newblock Reflection couplings and contraction rates for diffusions.
\newblock {\em Probability Theory and Related Fields}, pages 1--36, 2015.

\bibitem{eberle2016quantitative}
Andreas Eberle, Arnaud Guillin, and Raphael Zimmer.
\newblock Quantitative {H}arris type theorems for diffusions and
  {M}ckean-{V}lasov processes.
\newblock {\em arXiv preprint arXiv:1606.06012}, 2016.

\bibitem{Eberle2017}
Andreas Eberle, Arnaud Guillin, and Raphael Zimmer.
\newblock Couplings and quantitative contraction rates for {L}angevin dynamics.
\newblock {\em arXiv preprint arXiv:1703.01617}, 2017.

\bibitem{Fern}
E~Robert Fernholz.
\newblock Stochastic portfolio theory.
\newblock In {\em Stochastic Portfolio Theory}, pages 1--24. Springer, 2002.

\bibitem{halfin1981heavy}
Shlomo Halfin and Ward Whitt.
\newblock Heavy-traffic limits for queues with many exponential servers.
\newblock {\em Operations research}, 29(3):567--588, 1981.

\bibitem{harrison1981reflected}
J~Michael Harrison and Martin~I Reiman.
\newblock Reflected {B}rownian motion on an orthant.
\newblock {\em The Annals of Probability}, pages 302--308, 1981.

\bibitem{harrison1987brownian}
J~Michael Harrison and Ruth~J Williams.
\newblock Brownian models of open queueing networks with homogeneous customer
  populations.
\newblock {\em Stochastics: An International Journal of Probability and
  Stochastic Processes}, 22(2):77--115, 1987.

\bibitem{harwil87}
J~Michael Harrison and Ruth~J Williams.
\newblock Multidimensional reflected {B}rownian motions having exponential
  stationary distributions.
\newblock {\em The Annals of Probability}, pages 115--137, 1987.

\bibitem{IPS}
Tomoyuki Ichiba, Soumik Pal, and Mykhaylo Shkolnikov.
\newblock Convergence rates for rank-based models with applications to
  portfolio theory.
\newblock {\em Probability Theory and Related Fields}, 156(1-2):415--448, 2013.

\bibitem{ichiba2011hybrid}
Tomoyuki Ichiba, Vassilios Papathanakos, Adrian Banner, Ioannis Karatzas, and
  Robert Fernholz.
\newblock Hybrid {A}tlas models.
\newblock {\em The Annals of Applied Probability}, 21(2):609--644, 2011.

\bibitem{Karatzas}
Ioannis Karatzas and Steven Shreve.
\newblock {\em Brownian motion and stochastic calculus}, volume 113.
\newblock Springer Science \& Business Media, 2012.

\bibitem{klenke2013}
Achim Klenke.
\newblock {\em Probability theory: a comprehensive course}.
\newblock Springer Science \& Business Media, 2013.

\bibitem{meyn2012markov}
Sean~P Meyn and Richard~L Tweedie.
\newblock {\em Markov Chains and Stochastic Stability}.
\newblock Springer Science \& Business Media, 2012.

\bibitem{PP}
Soumik Pal, Jim Pitman, et~al.
\newblock One-dimensional brownian particle systems with rank-dependent drifts.
\newblock {\em The Annals of Applied Probability}, 18(6):2179--2207, 2008.

\bibitem{reiman1984open}
Martin~I Reiman.
\newblock Open queueing networks in heavy traffic.
\newblock {\em Mathematics of operations research}, 9(3):441--458, 1984.

\bibitem{reiman1977queueing}
Martin~Ira Reiman.
\newblock Queueing networks in heavy traffic.
\newblock Technical report, STANFORD UNIV CALIF DEPT OF OPERATIONS RESEARCH,
  1977.

\bibitem{RW2000_2}
L.~C.~G. Rogers and David Williams.
\newblock {\em {Diffusions, Markov Processes, and Martingales}}, volume~2.
\newblock Cambridge Mathematical Library, second edition, 2000.

\bibitem{stolyar1995stability}
Alexander~L Stolyar.
\newblock On the stability of multiclass queueing networks: {A} relaxed
  sufficient condition via limiting fluid processes.
\newblock {\em Markov Processes and Related Fields}, 1(4):491--512, 1995.

\bibitem{wang2014measuring}
Rob~J Wang and Peter~W Glynn.
\newblock Measuring the initial transient: {R}eflected {B}rownian motion.
\newblock In {\em Simulation Conference (WSC), 2014 Winter}, pages 652--661.
  IEEE, 2014.

\bibitem{williams1998diffusion}
Ruth~J Williams.
\newblock Diffusion approximations for open multiclass queueing networks:
  {S}ufficient conditions involving state space collapse.
\newblock {\em Queueing systems}, 30(1-2):27--88, 1998.

\end{thebibliography}

\end{document}